\documentclass[a4paper,11pt]{article}

\usepackage[T1]{fontenc}
\usepackage{lmodern}

\usepackage{amsmath, amssymb, amsthm}

\usepackage[margin=2.5cm]{geometry}
\usepackage[svgnames,x11names,table]{xcolor}
\usepackage{subcaption}
\usepackage{graphicx}
\usepackage{tikz}
\usetikzlibrary{positioning}
\usepackage{mathtools}

\usepackage{bbm}
\usepackage{framed}  
\usepackage[ruled,vlined]{algorithm2e}
\usepackage{booktabs}
\usepackage{cite}

\usepackage{etoolbox}
\makeatletter
\patchcmd{\@maketitle}{\LARGE \@title}{\LARGE\bfseries\@title}{}{}

\renewcommand{\@seccntformat}[1]{\csname the#1\endcsname.\quad}
\makeatother

\definecolor{darkblue}{rgb}{0,0,.5}

\usepackage{hyperref} 
\hypersetup{
        colorlinks=true,   % false: boxed links; true: colored links
        linkcolor=darkblue,    % color of internal links
        citecolor=darkblue,    % color of links to bibliography
        urlcolor=darkblue      % color of external links
      }

\usepackage{xfrac}
\usepackage{microtype}
\renewcommand{\leq}{\leqslant}
\renewcommand{\geq}{\geqslant}
\renewcommand{\emptyset}{\emptyset}

\let\emptyset\varnothing

\newtheorem{theorem}{Theorem}[section]
\newtheorem{lemma}[theorem]{Lemma}
\newtheorem{corollary}[theorem]{Corollary}
\newtheorem{proposition}[theorem]{Proposition}
\newtheorem{assumption}[theorem]{Assumption}
\theoremstyle{definition}
\newtheorem{definition}[theorem]{Definition}
\theoremstyle{definition}

\theoremstyle{definition}
\newtheorem{remark}[theorem]{Remark}

\usepackage[shortlabels]{enumitem}

\allowdisplaybreaks

\newcounter{step}[algocf]
\newcommand\step[1]{%
    \refstepcounter{step}    
    \vskip 0.25\baselineskip
    \ifx\hfuzz#1\hfuzz
        \noindent\(\triangleright\)~\textbf{Step~\arabic{step}.}%
    \else
        \noindent\(\triangleright\)~\textbf{Step~\arabic{step}}~(\texttt{#1})\textbf{.}%
    \fi
}

\newcommand{\weak}{\ensuremath{\:{\rightharpoonup}\;}}

\newcommand{\mE}{\mathcal{E}}

\newcommand{\mN}{\mathcal{N}}
\newcommand{\bone}{\ensuremath{{\boldsymbol{1}}}}
\newcommand{\Qname}{fixed-point relocator}
\newcommand{\be}{\mathbf{e}}

\newcommand{\Lip}{\mathcal{L}}

\newcommand{\tto}{\twoheadrightarrow}

\definecolor{myseagreen}{HTML}{3FBC9D}
\usepackage[color=myseagreen]{todonotes}

\newcommand{\R}{\mathbbm R}
\newcommand{\E}{\mathbbm E}
\newcommand{\N}{\mathbbm N}

\DeclareMathOperator{\diag}{diag}
\DeclareMathOperator{\Inc}{Inc}

\DeclareMathOperator{\proj}{proj}

\DeclareMathOperator{\gra}{gra}

\DeclareMathOperator{\zer}{zer}
\DeclareMathOperator{\Fix}{Fix}
\DeclareMathOperator{\range}{range}

\DeclareMathOperator{\Id}{Id}

\newcommand{\bz}{\mathbf{z}}
\newcommand{\bx}{\mathbf{x}}
\newcommand{\by}{\mathbf{y}}

\newcommand{\bu}{\mathbf{u}}
\newcommand{\bw}{\mathbf{w}}

\newcommand{\bp}{\mathbf{p}}
\newcommand{\bq}{\mathbf{q}}

\newcommand{\bA}{\mathbf{A}}
\newcommand{\bB}{\mathbf{B}}

\newcommand{\bS}{\mathbf{S}}

\newcommand{\bM}{\mathbf{M}}

\newcommand{\bF}{F}
\newcommand{\bN}{\mathbf{N}}
\newcommand{\bP}{\mathbf{P}}
\newcommand{\bR}{\mathbf{R}}

\newcommand{\kkk}{{k \in \N}}

\renewcommand{\Bar}[1]{\overline{#1}}

\newcommand{\vertiii}[1]{{\left\vert\kern-0.25ex\left\vert\kern-0.25ex\left\vert #1 
    \right\vert\kern-0.25ex\right\vert\kern-0.25ex\right\vert}}

\DeclareFontFamily{U}{mathx}{}
\DeclareFontShape{U}{mathx}{m}{n}{<-> mathx10}{}
\DeclareSymbolFont{mathx}{U}{mathx}{m}{n}
\DeclareMathAccent{\widehat}{0}{mathx}{"70}
\DeclareMathAccent{\widecheck}{0}{mathx}{"71}

\begin{document}

\title{Variable Stepsize Distributed Forward-Backward Splitting Methods as Relocated Fixed-Point Iterations\thanks{The research of FA, MND and MKT was supported in part by Australian Research Council grant DP230101749. The authors thank the anonymous referees for their insightful comments and constructive suggestions.}}
\author{Felipe Atenas\thanks{Centro de Modelamiento Matem\'atico (CNRS IRL2807), Universidad de Chile, Chile. E-mail:~\href{mailto:fatenas@cmm.uchile.cl}{fatenas@cmm.uchile.cl}}
,
Minh N.\ Dao\thanks{School of Science, RMIT University, Australia. E-mail:~\href{mailto:minh.dao@rmit.edu.au}{minh.dao@rmit.edu.au}},
and
Matthew K.\ Tam\thanks{School of Mathematics \& Statistics, The University of Melbourne, Australia. E-mail:~\href{mailto:matthew.tam@unimelb.edu.au}{matthew.tam@unimelb.edu.au}}
}

\date{\today}

\maketitle

\begin{abstract}
We present a family of distributed forward-backward methods with variable stepsizes to find a solution of structured monotone inclusion problems. The framework is constructed by means of relocated fixed-point iterations, extending the approach introduced in [SIAM J. Optim. 2026; 36(3):1297--1325] to conically averaged operators, thus including iteration operators for methods of forward-backward type devised by graphs. The family of methods we construct preserves the per-iteration computational cost and the convergence properties of their constant stepsize counterparts. Specifically, we show that the resulting methods generate a sequence that converges to a fixed point of the underlying iteration operator, whose shadow sequences converge to a solution of the problem. Numerical experiments illustrate the behaviour of our framework in structured sparse optimisation problems.
\end{abstract}


\paragraph{Mathematics Subject Classification (MSC 2020):}
%90C25,	%Convex programming
%90C30, %Nonlinear programming
68W15,  %Distributed algorithms
49M27,	%Decomposition methods
47H05,  %Monotone operators and generalizations
47H10,  %Fixed-point theorems
47H04, 	%Set-valued operators
47H09.   %Contraction-type mappings, nonexpansive mappings, $A$-proper mappings,

\section{Introduction and motivation}
		
		Many modern applications in statistics \cite{zou2005regularization,tibshirani2005sparsity}, machine learning \cite{candes2012exact}, and optimisation \cite{condat2023proximal}, can be abstractly modelled as a structured inclusion problem of the form \begin{equation} \label{e:distributed-monotone-inclusion} 
			\text{find~} x \in X \text{~such that~} 0 \in \sum_{i=1}^nA_ix + \sum_{j=1}^p B_jx,
		\end{equation} where $X$ is a real Hilbert space,  $A_i: X \rightrightarrows X $ is a set-valued operator for $i =1, \dots, n$, and $B_j: X \to X$ is a single-valued operator for $j=1,\dots, p$. Splitting methods exploit the structure of problem \eqref{e:distributed-monotone-inclusion} by performing separate operations on the model components. The family of splitting methods includes the classical \emph{forward-backward} and \emph{Douglas--Rachford} methods \cite{Lions1979SplittingAF,Passty1979ErgodicCT,Bruck1975AnIS} that can solve, respectively, the cases $n=1$ and $p=1$, and $n=2$ and $p=0$. A unifying manner to encode this family of methods is through \emph{fixed-point iterations}. For instance, the forward-backward method defines its iterates via \begin{equation} \label{e:fixed-point-iteration}
			z_{n+1} = T_\gamma z_n \quad \text{for all ~} n \in \N,
		\end{equation}  where $\gamma >0$ is a parameter, and $T_\gamma := J_{\gamma A_1}(\Id - \gamma B_1)$ is the corresponding iteration operator, $J_{\gamma A_1} = (\Id + \gamma A_1)^{-1}$ is the \emph{resolvent} of $\gamma A_1$, and $\Id$ denotes the identity operator on $X$. When $n=2$ and $p=1$ in \eqref{e:distributed-monotone-inclusion}, the \emph{Davis--Yin three-operator splitting method} \cite{davis2017three} generates its iterates via \eqref{e:fixed-point-iteration} with \begin{equation}\label{e:DY-operator}
			T_\gamma = \Id + J_{\gamma A_2}\big( 2 J_{\gamma A_1} - \Id - \gamma B_1 J_{\gamma A_1} \big) - J_{\gamma A_1}.
		\end{equation}  These methods enjoy convergence guarantees under mild regularity assumptions, yielding, in the limit, a point $z \in X$ such that $z = T_\gamma z$, that is, a \emph{fixed point} of the operator $T_\gamma$. In the case of the forward-backward method, $z$ is a solution to \eqref{e:distributed-monotone-inclusion}, and for the Davis--Yin method, a solution to \eqref{e:distributed-monotone-inclusion} can be recovered by computing $x = J_{\gamma A_1}z$. The parameter $\gamma >0$, also known as the \emph{stepsize}, plays a crucial role in controlling the convergence of the methods. When the stepsize $\gamma > 0$ is fixed, traditional convergence guarantees require this parameter to lie in an interval of the form $(0, \frac{2}{L})$, where $L > 0$ is a constant measuring some regularity condition of $B_1$, usually \emph{Lipschitz continuity} or \emph{cocoercivity}. Once convergence is established, the next natural question is  efficiency. Improvements in the performance of the methods can be achieved, for example, by allowing the stepsize parameter $\gamma >0$ to vary throughout the iterations. Another approach that varies a different type of parameters is the \emph{Fast Iterative Shrinkage-Thresholding Algorithm} (FISTA) \cite{beck2009fast}, and  \emph{Nesterov's acceleration} \cite{nesterov1983method}, based on the idea of \emph{momentum}.
		
		In this work, we present a systematic framework that allows variable stepsizes while preserving convergence guarantees of methods of forward-backward type to solve problem \eqref{e:distributed-monotone-inclusion}.   This unifying theory is realised by means of \emph{relocated fixed-point iterations},  originally devised in \cite{atenas2025relocated} for averaged nonexpansive operators. In this work, we show that the framework can be extended to the \emph{conically averaged} setting, providing more flexibility to the selection of parameters. We apply the theory to obtain convergent variable stepsize multioperator forward-backward methods, extending the results originally presented in \cite{atenas2025relocated} for variable stepsize resolvent splitting methods.  As far as the authors know, this is the first work that establishes convergence of multioperator forward-backward splitting methods with variable stepsizes. 
		
		A key ingredient of the analysis is the concept of a \emph{fixed-point relocator} (see Definition~\ref{d:FPR}). An archetypal example  of a fixed-point relocator is the one associated with the \emph{Douglas--Rachford} operator. For $n=2$ and $p=0$ in \eqref{e:distributed-monotone-inclusion}, and $\gamma >0$ fixed, the Douglas--Rachford algorithm defines iterates via \eqref{e:fixed-point-iteration} with \begin{equation}\label{e:DR-operator}
			T_\gamma = \Id + J_{\gamma A_2}\big( 2 J_{\gamma A_1} - \Id  \big) - J_{\gamma A_1},
		\end{equation} that is, \eqref{e:DY-operator} with $B_1 = 0$. A fixed-point relocator of the operator in \eqref{e:DR-operator} is given, for all $\delta, \gamma \in \R_{++}$, by \begin{equation} \label{e:FPR-DR}
			Q_{\delta\gets\gamma} := \frac{\delta}{\gamma}\Id + \left(1 - \frac{\delta}{\gamma}\right) J_{\gamma A_1}.
		\end{equation} The operator in \eqref{e:FPR-DR} has the following two properties: $Q_{\delta\gets\gamma}^{-1} = Q_{\gamma\gets\delta}$, and $Q_{\delta\gets\gamma}(\Fix T_\gamma) = \Fix T_\delta$ \cite[Lemma~3.1, Theorem~3.5]{atenas2025relocated}. The second property motivates the name fixed-point relocator: for any $x \in \Fix T_\gamma$, $Q_{\delta\gets\gamma}x \in \Fix T_\delta$. The associated relocated fixed-point iteration thus is \begin{equation} \label{e:reloc-iterations}
			z_{k+1} = Q_{\gamma_{k+1}\gets\gamma_k}T_{\gamma_k}z_k \text{~for all~} k \in \N.
		\end{equation} As we shall see in \eqref{e:FPR-DY}, the operator in \eqref{e:FPR-DR} is also a fixed-point relocator of the Davis--Yin operator \eqref{e:DY-operator}, which yields the relocated three-operator scheme in Algorithm~\ref{a:DY}.
		
		This paper is organised as follows. In Section~\ref{s:prelim} we define the setting of conically averaged operators, and extend (a version of) the so-called \emph{demiclosedness principle} to this family of operators, crucial for the posterior convergence analysis. In Section~\ref{s:relocated-conic} we present the main components of the fixed-point relocator framework from \cite{atenas2025relocated}, and its version for conically averaged operators. We continue in Section~\ref{s:FB} with the description of a family of relocated forward-backward methods to solve problems of the form \eqref{e:distributed-monotone-inclusion}, and in Section~\ref{s:graph} we examine graph-based forward-backward methods for multioperator inclusions, including the relocated Davis--Yin three-operator splitting algorithm as a special case. We also investigate different fixed-point relocators depending on the graph structure used to devise these methods, with the property that they maintain the per-iteration computational cost. Finally, we conclude in Section~\ref{s:conclusion}  with some final remarks and future research directions.

\section{Preliminaries} \label{s:prelim}
		
		We denote the set of positive real numbers by $\R_{++}$. For $\alpha \in \R$, we denote by $\alpha_+ = \max\{0,\alpha\}$ the positive part of $\alpha$. For any matrix $M \in \R^{n \times p}$, we use the notation $\bM := M \otimes \Id$, where $\otimes$ denotes the Kronecker product. The bold vector $\bone$ denotes the vector of $1$'s, whose dimension will depend on the context. We also denote by $M^\dagger$ the Moore--Penrose pseudoinverse of $M$. For a vector $v \in \R^d$, $\diag(v) \in \R^{d \times d}$ denotes the diagonal matrix whose diagonal entries are the coordinates of the vector $v$, namely, for all $i = 1,\dots, d$, $\diag(v)_{ii} = v_i$, and $0$ otherwise.
		
		Let $X$ be a real Hilbert space endowed with inner product $\left< \cdot, \cdot \right>$ and induced norm $\|\cdot\|$. We use  $X^{\text{strong}}$ to denote the space $X$ endowed with the norm topology, and $X^{\text{weak}}$ to denote $X$ endowed with the weak topology. As is customary, $\to$ denotes convergence in the strong topology, while $\weak$ denotes convergence in the weak topology. Given $n \geq 1$, the elements in the space $X^n := \Pi_{i=1}^n X$ will be denoted with bold symbols, namely, $\bx = (x_1, \dots, x_n)$. We endow the space $X^n$ with the $\ell_2$-norm, that is, for all $\bx \in X^n$, $\|\bx\| = \sqrt{\sum_{i=1}^n \|x_i\|^2}$, which results in $X^n$ being a Hilbert space itself. We do not make a distinction in the notation for the norm in $X$, $X^n$ and the matrix norm, since it is clear from the context which one we refer to.
		
		Given a set-valued operator $A: X \rightrightarrows X$,  its \emph{graph} is denoted by $\gra A := \{ (x,y) \in X\times X: y \in Ax\}$, its set of \emph{zeros} by $\zer A := \{x \in X : 0 \in Ax\}$, and its set of fixed points by $\Fix A := \{x \in X: x \in Ax\}$. We also denote the \emph{inverse} of $A$ by $A^{-1}: X \rightrightarrows X$ which is defined by $\gra A^{-1} := \{(y,x) \in X \times X: y \in Ax\}$. We say $A$ is \emph{monotone} if for all
		$(x,y), (u,v) \in \gra A$, $\left< x-u, y - v\right> \geq 0,$
		and \emph{maximally monotone} if it has no proper monotone extension.  For a maximally monotone operator $A$ on $X$, the resolvent of $A$, $J_A := (\Id + A)^{-1}$, is well-defined on $X$ and single-valued \cite[Proposition 23.8]{BC2017}.
		\begin{lemma}[Properties of resolvents {\cite[Lemma~3.1, Proposition~3.4]{atenas2025relocated}}] \label{l:known-J}
			Let $A$ be a maximally monotone operator on $X$. Then
			\begin{enumerate}
			\item\label{l:known-J-1} For all $\delta,\gamma \in \R_{++}$, $J_{\delta A}\Big(\tfrac{\delta}{\gamma}\Id 
				+ \big(1-\tfrac{\delta}{\gamma}\big)J_{\gamma A} \Big) = J_{\gamma A}$.
				\item\label{l:known-J-2}  $ X \times (0, +\infty) \ni (x,\gamma) \mapsto J_{\gamma A}x$ is (strongly) continuous.
			\end{enumerate}
		\end{lemma}
		
		Given a nonempty closed convex set $C \subseteq X$, $P_C:X\to C$ denotes the \emph{projection operator} onto $C$, that is, the mapping such that for all $x \in X$, $\min_{y \in C} \|y-x\| = \|P_Cx-x\|$. The mapping $N_C: X \rightrightarrows X $, given by $N_C(x) := \{ v \in X: \langle v, y -x \rangle \leq 0 \text{~for all~} y \in C \}$ if $x \in C$, and $N_C(x) = \varnothing$ otherwise, is the \emph{normal cone} of $C$, and satisfies $P_C = J_{N_C}$.
		
		We use  $T: X \to X$ to denote a single-valued operator on $X$. We say a single-valued operator $T$ on $X$ is $L$-Lipschitz continuous with constant $L>0$, if  for all $x,y \in X$, $ \|Tx - Ty\| \leq L \|x-y\|$.  We say $T$ is \emph{nonexpansive} if $T$ is $1$-Lipschitz continuous. For $\beta >0$, a single-valued operator $T$ on $X$ is said to be $\beta$-\emph{cocoercive} if for all $x,y \in X$, $\left< Tx - Ty, x-y \right> \geq \frac{1}{\beta}\|Tx-Ty\|^2$.   In particular, if $T$ is $\beta$-cocoercive, then $T$ is monotone and $\beta$-Lipschitz continuous, and thus maximally monotone \cite[Corollary 20.28]{BC2017}. We say $T$ is \emph{conically} $\eta$\emph{-averaged} \cite[Proposition 2.2]{bartz2022conical} for $\eta > 0$, if for all $x,y \in X$, \begin{equation*}
			\|Tx - Ty\|^2  + \frac{1-\eta}{\eta}\|(\Id-T)x - (\Id-T)y\|^2 \le \|x-y\|^2.
		\end{equation*} Equivalently, an operator $T$ on $X$ is conically $\eta$-averaged  if  there exists a nonexpansive operator $N$ on $X$ such that $T = (1-\eta)\Id + \eta N$ (see \cite[Definition 2.1]{bartz2022conical} and \cite[Definition 3.1]{bauschke2021generalized}). When $\eta \in (0,1)$, $T$ is said to be $\eta$-averaged nonexpansive. For a maximally monotone operator $A$ on $X$, its resolvent $J_A$ is  $\frac{1}{2}$-averaged nonexpansive. Iteration operators of forward-backward methods are conically averaged, see Lemma~\ref{l:con-ave} below.
		
		Let $\E$ be a finite-dimensional Euclidean space.  The \emph{demiclosedness principle} is a fundamental tool in fixed-point theory used in the convergence analysis of fixed-point iterations. An operator $T$ on $X$ is \emph{demiclosed} if $\gra T$ is sequentially closed in $X^{\text{weak}}  \times X^{\text{strong}}$ (see, e.g. \cite[Definition 4.26]{BC2017}). The demiclosedness principle states that for every nonexpansive operator $T$, the operator $\Id - T$ is demiclosed \cite{browder1968semicontractive}.  For a parametrised family of operators, a generalisation of the demiclosedness principle was presented in \cite{atenas2025relocated}. First, we recall the definition of a \emph{parametric demiclosed} family of operators, and then present the \emph{parametric demiclosedness principle}. Originally, the results were shown for parameters in $\E = \R$, but they can be directly extended for any finite-dimensional Euclidean space $\E$, as in the latter case the strong and weak topologies coincide.
		
		\begin{definition}[Parametric demiclosedness]
        Let $D \subseteq X$ be a nonempty weakly sequentially closed set, and $\Gamma \subseteq \E$ be a nonempty set. We say an operator $G: D \times \Gamma \to X$ is \emph{parametrically demiclosed}~if $\gra G$ is sequentially closed in $(X^{\text{weak}} \times \E) \times X^{\text{strong}}$.
		\end{definition}
		
		\begin{theorem}[Parametric demiclosedness principle~{\cite[Theorem 3.9]{atenas2025relocated}}] \label{t:edemi}
			Let $D \subseteq X$ be a nonempty weakly sequentially closed set,  $\mathcal{P}\subseteq\E$ be a nonempty closed set, and $(T_{p})_{p\in\mathcal{P}}$ be a family of {nonexpansive} operators from $D$ to $X$, such that $(\forall x \in D)\; p \mapsto T_{p}x$ is (strongly) continuous. Then the operator
			\begin{align*}
				D \times \mathcal{P} &\to X\colon(x,p)\mapsto (\Id-T_{p})x
			\end{align*} 
			is parametrically demiclosed.
		\end{theorem}
		
		\begin{remark} \label{r:parametric-DCP}
			In the context of Theorem~\ref{t:edemi}, suppose there exist a sequence  $(x_k)_\kkk $ in $D$ converging weakly to $x \in D$, and a sequence $(p_k)_\kkk$ in $\mathcal{P}$ converging to $\overline{p} \in \mathcal{P}$. If $x_k-T_{p_k}x_k \to 0$, then the parametric demiclosedness principle implies that $x \in \Fix  T_{\overline{p}}$. 
		\end{remark}
		
		Theorem~\ref{t:edemi} requires the family of operators to be nonexpansive. A demiclosedness principle for conically averaged operators is shown in \cite[Theorem 3.3]{bartz2020demiclosedness}, although not for the parametric case. In the next result, we extend the parametric demiclosedness principle to the class of conically averaged operators.

		\begin{corollary}[Parametric demiclosedness principle for conically averaged operators] \label{l:parametric-DCP-extended}
			Let $\mathcal{P} \subseteq \E$ be a nonempty closed set, $D \subseteq X$ be a nonempty weakly sequentially closed set, and $(T_{p})_{p\in\mathcal{P}}$ be operators on $X$. Suppose there exists a continuous function $\eta: \mathcal{P} \to \R_{++}$ such that for all $p\in\mathcal{P}$, the operator $T_{p}$ is conically $\eta(p)$-averaged. Moreover, suppose that  for all $z \in X$, the mapping $p  \mapsto T_{p}z$ is (strongly) continuous. Then, the operator   \begin{align*}
				D \times \mathcal{P} &\to X\colon(x,p)\mapsto (\Id-T_{p})x
			\end{align*} 
			is parametrically demiclosed.
		\end{corollary}
		
		\begin{proof}
			In view of the definition of conical averagedness, for all $p\in\mathcal{P}$, there exists a nonexpansive operator $N_{p}$ on $X$ such that $T_{p} = (1-\eta(p)) \Id + \eta(p) N_{p}$.  First, we show that for any $z \in X$,  the mapping $p \in\mathcal{P} \mapsto N_{p}z$ is strongly continuous. Let    $p_k \to p \in \mathcal{P}$, then from continuity of $p \mapsto T_{p}z$, it follows that \[(1-\eta(p_k)) z + \eta(p_k) N_{p_k} z = T_{p_k}z \to T_{p}z.\] Furthermore,  continuity of $p \mapsto \eta(p)$ implies  $\eta(p_k) \to \eta(p)$, and thus \[\eta(p_k) N_{p_k} z = T_{p_k}z - (1-\eta(p_k)) z  \to T_{p}z - (1-\eta(p)) z = \eta(p) N_{p}z .\] Since $0 \notin \eta(\mathcal{P})$, which follows directly from the definition of conical averagedness, then $N_{p_k} z \to N_{p}z$, and thus the claim holds. Next, suppose $z_k \weak z$, $p_k \to p$, and $z_k - T_{p_k}z_k \to u$. Our goal is to show that $z - T_{p}z = u$. Indeed, since  $\eta(p_k)(z_k - N_{p_k} z_k) = z_k - T_{p_k}z_k \to u$ and $0 \notin \eta(\mathcal{P})$, then $ z_{k} - N_{p_{k}}z_{k}\to \frac{u}{\eta(p)}$. From Theorem~\ref{t:edemi}, $(N_{p})_{p\in\mathcal{P}}$ is parametrically demiclosed, then $\frac{1}{\eta(p)}(z - T_pz) =z - N_{p}z  = \frac{u}{\eta(p)}$, from where the conclusion follows.\end{proof}
		
		One of the main contributions of \cite{atenas2025relocated} is developing a convergence framework of parametrised fixed-point iterations, called {relocated fixed-point iterations}, without resorting to the classical argument of Fej{\'e}r monotonicity. Instead, it is based on the concept of an \emph{Opial sequence} \cite{arakcheev2025opial}.  A sequence $(x_k)_\kkk$ in $X$ is \emph{Opial} with respect to a nonempty set $C \subseteq X$, if for all $c \in C$, $\lim_\kkk \|x_k - c\|$ exists. This notion  captures essential properties of Fej{\'e}r monotone sequences without requiring monotonicity. Because of the ``lack of structure'' in the definition due to nonmonotonicity, showing that the property holds might be challenging in general. In the following result, we show a practical, constructive way of using this property that fits the algorithmic application in our present work. 
		
		\begin{lemma}\label{l:Opial}
			Let $(z_k)_\kkk$ be a sequence in $X$, and $C \subseteq X$ a nonempty set. Suppose for all $c \in C$, there exists a sequence $(c_k)_\kkk$ in $X$ converging to $c$, such that $\lim_\kkk \|z_k - c_k\|$ exists. Then, the following statements are true.
			\begin{enumerate}
				\item\label{p:Opial-bounded}
				$(z_k)_\kkk$ is bounded. 
				\item\label{p:Opial-cluster}
				$(z_k)_\kkk$ converges weakly to a point in $C$ if and only if all weak cluster points of $(z_k)_\kkk$ lie in~$C$.
			\end{enumerate}
		\end{lemma}
		
		\begin{proof}
			Using the triangle inequality, it holds for all $\kkk$, \begin{equation*}
				\|z_k-c_k\| - \|c_k-c\| \le \|z_k-c\| \le \|z_k-c_k\| +  \|c_k-c\|.
			\end{equation*} Taking the limit as $k \to +\infty$, it follows that $\lim_\kkk \|z_k-c\| = \lim_\kkk \|z_k-c_k\|$, and thus $(z_k)_{\kkk}$ is Opial with respect to $C$. The results follow from \cite[Proposition 2.1, Corollary 2.3]{arakcheev2025opial}.\end{proof}
		
		We conclude this section by recalling a technical lemma to be used in the next section.
		
		\begin{lemma}[Robbins--Siegmund, {\cite[Section 2.2.2, Lemma~11]{polyak1987introduction}}]
			\label{f:RS}
			Let $(\alpha_k)_\kkk,(\beta_k)_\kkk,(\varepsilon_k)_\kkk $ be sequences in $\R_{+}$
			such that 
			$\sum_\kkk\varepsilon_k<+\infty$ and for all $\kkk,\; \alpha_{k+1} \leq 
			(1+\varepsilon_k)\alpha_k -\beta_k.$ Then $(\alpha_k)_\kkk$ converges and 
			$\sum_\kkk \beta_k<+\infty$.
		\end{lemma}

%%%%%%%%%%%%%%%%%%%%%%%%%%%%%%%%%%%%%%%%%%%%%%%%%%%%%%%%%%%%%%%%%%%

\section{Relocated fixed-point iterations of conically averaged operators} \label{s:relocated-conic}

In this section, we present the framework of relocated fixed-point iterations for conically averaged operators. We start by defining the concept of a fixed-point relocator. Given a family of operators $(T_\gamma)_{\gamma \in \R_{++}}$,   an operator is a fixed-point relocator if, in particular, it transforms (relocates) fixed points of $T_\gamma$ to fixed points of $T_\delta$ for all $\gamma,\delta \in \R_{++}$. This property is crucial to guide the convergence of relocated fixed-point iterations in the form of \eqref{e:reloc-iterations}. 

\begin{definition}[Fixed-point relocator] \label{d:FPR}
	Let $\Gamma \subseteq \R_{++}$ be a nonempty set, and let $(T_\gamma)_{\gamma \in \Gamma}$ be a family of operators on $X$. We say that the family of operators $(Q_{\delta\gets \gamma})_{\gamma,\delta \in \Gamma}$ on $X$ defines \emph{\Qname{s}} for $(T_\gamma)_{\gamma \in \Gamma}$ with Lipschitz constants $(\Lip_{\delta\gets \gamma})_{\gamma,\delta \in \Gamma}$ in $[1,+\infty[$ if the following hold.
	\begin{enumerate}
		\item\label{a:trans_bijection}
		For all $ \gamma, \delta \in \Gamma$, 
		$Q_{\delta\gets \gamma}|_{\Fix T_\gamma}$ is a bijection from $\Fix T_\gamma$ to $\Fix T_\delta$.
		\item\label{a:trans_cont} For all $ \gamma \in \Gamma, \; x\in\Fix T_\gamma$, $\Gamma \ni \delta\mapsto Q_{\delta\gets \gamma}x \in X$ is continuous. 
		\item\label{a:trans_semigroup}
		For all $ \gamma, \delta, \varepsilon \in \Gamma, \; x \in \Fix T_{\gamma}$, $Q_{\varepsilon\gets \delta}Q_{\delta\gets \gamma}x = Q_{\varepsilon\gets \gamma}x$.
		\item\label{a:trans_nonex-type}
		For all $ \gamma, \delta \in \Gamma$, $Q_{\delta\gets \gamma}$ is $\Lip_{\delta\gets \gamma}$-Lipschitz continuous.
	\end{enumerate}
\end{definition}
		
\begin{remark}[Consequences of the \Qname{} definition] \label{r:FPR-inverse}
	In view of \cite[Remark 4.3]{atenas2025relocated}, given \Qname{s} $(Q_{\delta\gets \gamma})_{\gamma,\delta \in \Gamma}$ for $(T_\gamma)_{\gamma\in\R_{++}}$, the following holds for all $\delta,\gamma \in \R_{++}$, \begin{equation*}
		Q_{\gamma\gets\gamma} = \Id \text{~on~} \Fix T_{\gamma}, \text{~and~} (Q_{\delta\gets \gamma}|_{\Fix T_\gamma})^{-1} = Q_{\gamma\gets\delta}|_{\Fix T_\delta}.
	\end{equation*} 
\end{remark}		
\begin{remark}[Multiple fixed-point relocators] \label{r:FPR} As discussed in \cite[Remark 4.2]{atenas2025relocated}, given  \Qname{s} $(Q_{\delta\gets \gamma})_{\gamma,\delta\in\R_{++}}$ for $(T_\gamma)_{\gamma\in\R_{++}}$, observe that Definition~\ref{d:FPR}~\ref{a:trans_bijection}--\ref{a:trans_semigroup} only need to hold on $\Fix T_\gamma$. Therefore, any Lipschitz continuous operator $\tilde{Q}_{\delta\gets \gamma}$  such that $Q_{\delta\gets \gamma}|_{\Fix T_\gamma} = \tilde{Q}_{\delta\gets \gamma}|_{\Fix T_\gamma}$ also defines a \Qname~for $(T_\gamma)_{\gamma\in\R_{++}}$.
\end{remark}	

In the rest of this section, we will consider a family of operators $(T_{\theta,\gamma})_{(\theta,\gamma)\in\Theta\times\Gamma}$ parametrised on $\Theta\times\Gamma \subseteq \R_{++}^2$, where only the parameter $\gamma \in \Gamma$ determines the \Qname{s} from Definition~\ref{d:FPR}, while $\theta$ represents a parameter that does not modify the set of fixed points of $T_{\theta,\gamma}$. In Section~\ref{s:FB}, $\gamma$ plays the role of a stepsize, and $\theta$ may be interpreted as a relaxation parameter. In \eqref{e:z-iteration} below, we define a relaxed relocated fixed-point iteration scheme, where $(\lambda_k)_{\kkk}$ is a sequence of relaxation parameters of the relocated fixed-point dynamics (independently of the interpretation of the pair $(\theta,\gamma)$ in any given setting).

In the following, we introduce the sequence $(c_k)_\kkk$ constructed in \cite[Theorem 4.5]{atenas2025relocated}, which combined with Lemma~\ref{l:Opial} is the basis of the convergence analysis of the relocated fixed-point iterations of conically averaged operators.

\begin{lemma}\label{l:aux-seq}
	Let $\Theta,\Gamma \subseteq \R_{++}$ be nonempty, and $(T_{\theta,\gamma})_{(\theta,\gamma)\in \Theta \times \Gamma}$ be a family of operators on $X$. Let $\Bar{\theta} \in \Theta$, such that for all $(\theta,\gamma) \in \Theta \times \Gamma$, $\Fix T_{\theta,\gamma} = \Fix T_{\Bar{\theta},\gamma} \neq\varnothing$. Let $(Q_{\delta\gets \gamma})_{\gamma,\delta \in \Gamma}$ be operators on $X$ satisfying Definition~\ref{d:FPR}\ref{a:trans_bijection}-\ref{a:trans_semigroup} for $(T_{\Bar{\theta},\gamma})_{\gamma\in \Gamma}$. Let $(\theta_k)_\kkk $ be a sequence in $\Theta$, and  $(\gamma_k)_\kkk $ be a sequence in $\Gamma$ that converges to $\overline{\gamma} \in \Gamma$. Choose  $c_0 \in \Fix T_{\theta_0,\gamma_0}$, and for all $\kkk$, set $$c_{k+1} := Q_{\gamma_{k+1}\gets\gamma_k} c_k.$$ Then we have $c_k \in \Fix T_{\theta_k,\gamma_k}$ for all $\kkk$ and $(c_k)_\kkk$ converges strongly to $Q_{\overline{\gamma}\gets\gamma_0} c_0 \in \Fix T_{\Bar{\theta},\Bar{\gamma}}$. 
\end{lemma}
\begin{proof}
	We first argue by induction. The base case follows from construction. Suppose for $\kkk$, $c_k \in \Fix T_{\theta_k,\gamma_k}$, and thus $c_k \in \Fix T_{\Bar{\theta},\gamma_k}$. Then, from Definition~\ref{d:FPR}\ref{a:trans_bijection}, $c_{k+1} = Q_{\gamma_{k+1}\gets\gamma_k} c_k \in \Fix T_{\Bar{\theta},\gamma_{k+1}} = \Fix T_{\theta_{k+1},\gamma_{k+1}}$. Hence, the first claim holds. Next, for $\kkk$, Definition~\ref{d:FPR}\ref{a:trans_semigroup} yields \begin{equation*}
		c_k = Q_{\gamma_k\gets\gamma_{k-1}}c_{k-1} = Q_{\gamma_k\gets\gamma_{k-1}}Q_{\gamma_{k-1}\gets\gamma_{k-2}} \dots Q_{\gamma_1\gets\gamma_0}c_0 = Q_{\gamma_k\gets\gamma_0}c_0.
	\end{equation*} Using Definition~\ref{d:FPR}\ref{a:trans_cont}, it follows that $c_k \to Q_{\Bar{\gamma}\gets\gamma_0}c_0$, and since $c_0 \in \Fix T_{\Bar{\theta},\gamma_0}$,  Definition~\ref{d:FPR}\ref{a:trans_bijection}  yields $Q_{\Bar{\gamma}\gets\gamma_0}c_0 \in\Fix T_{\Bar{\theta},\Bar{\gamma}}$.\end{proof}

In the following, we present the main result of this section that establishes convergence of (relaxed) relocated fixed-point iterations of conically averaged operators (cf. \cite[Theorem 4.5]{atenas2025relocated}).

\begin{theorem}[Convergence of relocated fixed-point iterations for conically averaged operators] \label{t:FPR-convergence}
	Let $\Theta, \Gamma \subseteq \R_{++}$ be nonempty closed sets. For all $(\theta,\gamma) \in \Theta \times \Gamma$, let $T_{\theta,\gamma}$ be a conically $\eta(\theta,\gamma)$-averaged operator for $\eta: \Theta\times\Gamma \to \R_{++}$ continuous,  such that for some $\Bar{\theta} \in \Theta$, for all $\gamma \in \Gamma$, $\Fix T_{\theta,\gamma} = \Fix T_{\Bar{\theta},\gamma} \neq\varnothing$. Let $(Q_{\delta\gets \gamma})_{\gamma,\delta \in \Gamma}$ be  \Qname{s} for $(T_{\Bar{\theta},\gamma})_{\gamma \in \Gamma}$ with Lipschitz constants $(\Lip_{\delta\gets \gamma})_{\gamma,\delta \in \Gamma}$ in $[1,+\infty[$. Furthermore, let \begin{equation*}(\gamma_k)_\kkk \subseteq \Gamma \text{~be a sequence that 
			converges to~} \overline{\gamma}\in \Gamma
		\text{~and~}   \sum_\kkk (\Lip_{\gamma_{k+1}\gets \gamma_k} -1) < +\infty.
	\end{equation*} Let  $(\lambda_k)_\kkk $ and $(\theta_k)_\kkk $ be sequences in $\R_{++}$ and $\Theta$, respectively. For all $\kkk$, let $\eta_k := \eta(\theta_k,\gamma_k)$ be the parameter of conical averagedness of the operator $T_{\theta_k,\gamma_k}$, such that $\liminf_{\kkk}\eta_k^{-1}\lambda_k (1 - \eta_k\lambda_k)> 0$. Let $z_0 \in X$ and generate  $(z_k)_\kkk$ as follows: for all $  \kkk$, set
	\begin{equation} \label{e:z-iteration}
		z_{k+1} := Q_{\gamma_{k+1}\gets \gamma_k}\big((1-\lambda_k) z_k + \lambda_k T_{\theta_k,\gamma_k}z_k\big).
	\end{equation} Then, the following hold.
	
	\begin{enumerate}
		\item \label{t:FPR-convergence-ii} The sequence $(z_k)_\kkk$ is bounded, and $(z_k-T_{\theta_k,\gamma_k}z_k)_\kkk$ converges strongly to zero. 
		
		\item \label{t:FPR-convergence-iii} In addition, suppose for all $z \in X$ that the mapping $\Theta \times \Gamma \ni (\theta,\gamma) \mapsto T_{\theta,\gamma}z$ is (strongly) continuous, and that $(\theta_k)_\kkk$ is bounded. Then both $(z_k)_\kkk$ and $(T_{\theta_k,\gamma_k}z_k)_\kkk$ converge weakly to the same point in $\Fix T_{\Bar{\theta},\overline{\gamma}}$.
	\end{enumerate}
\end{theorem}
\begin{proof}
	Let $(c_k)_\kkk$ be the  sequence defined in Lemma~\ref{l:aux-seq}. For all $\kkk$, using Definition~\ref{d:FPR}\ref{a:trans_nonex-type}, $c_k \in \Fix T_{\theta_k,\gamma_k}$, and conical $\eta_k$-averagedness of $T_{\theta_k,\gamma_k}$, the following holds \begin{align*}
		 \|z_{k+1} - c_{k+1}\|^2 
		& = \| Q_{\gamma_{k+1}\gets \gamma_k}\big((1-\lambda_k) z_k + \lambda_k T_{\theta_k,\gamma_k}z_k) - Q_{\gamma_{k+1}\gets\gamma_k} c_k\|^2\\
		& \leq \Lip_{\gamma_{k+1}\gets\gamma_k}^2\|(1-\lambda_k) z_k + \lambda_k T_{\theta_k,\gamma_k}z_k - c_k\|^2\\
		& = \Lip_{\gamma_{k+1}\gets\gamma_k}^2\|(1-\lambda_k) (z_k - c_k) + \lambda_k (T_{\theta_k,\gamma_k}z_k - c_k) \|^2\\
		& = \Lip_{\gamma_{k+1}\gets\gamma_k}^2\big((1-\lambda_k)\|z_k - c_k \|^2 + \lambda_k \|T_{\theta_k,\gamma_k}z_k - c_k \|^2 - \lambda_k(1-\lambda_k)\|z_k - T_{\theta_k,\gamma_k}z_k\|^2\big)\\
		&\leq \Lip_{\gamma_{k+1}\gets\gamma_k}^2\big((1-\lambda_k)\|z_k - c_k \|^2 + \lambda_k \left(\|z_k - c_k\|^2 -\frac{1 - \eta_k}{\eta_k}\|T_{\theta_k,\gamma_k}z_k - z_k\|^2\right)\\ & \quad  - \lambda_k(1-\lambda_k)\|z_k - T_{\theta_k,\gamma_k}z_k\|^2\big)\\
		&= \Lip_{\gamma_{k+1}\gets\gamma_k}^2\|z_k - c_k\|^2 - \Lip_{\gamma_{k+1}\gets\gamma_k}^2\lambda_k\frac{1 - \eta_k\lambda_k}{\eta_k}\|  z_k - T_{\theta_k,\gamma_k}z_k\|^2.
	\end{align*} 
	Since $\liminf_\kkk \lambda_k\frac{1 - \eta_k\lambda_k}{\eta_k} >0$, there exists $k_0 \geq 1$, such that for all $k \geq k_0$, $1 - \eta_k\lambda_k > 0$. Hence, for $k \geq k_0$, we have\begin{equation*}
		\|z_{k+1} - c_{k+1}\| \leq   \Lip_{\gamma_{k+1}\gets\gamma_k}\|z_k - c_k\|.
	\end{equation*} 
	Using Lemma~\ref{f:RS} with $\alpha_k = \|z_k - c_k\|$, $ \varepsilon_k = \Lip_{\gamma_{k+1}\gets\gamma_k} -1$ and $\beta_k = 0$, we have that $(\|z_k - c_k\|)_\kkk$ converges. In view of Lemma~\ref{l:Opial}\ref{p:Opial-bounded} and Lemma~\ref{l:aux-seq}, $(z_k)_\kkk$ is bounded, and thus its set of cluster points is nonempty. Furthermore, since for all $k \geq k_0$, \begin{equation*}
		\Lip_{\gamma_{k+1}\gets\gamma_k}^2\lambda_k\frac{1 - \eta_k\lambda_k}{\eta_k}\| z_k - T_{\theta_k,\gamma_k}z_k\|^2  \leq   \Lip_{\gamma_{k+1}\gets\gamma_k}^2\|z_k - c_k\|^2 - \|z_{k+1} - c_{k+1}\|^2,
	\end{equation*} taking the limit as $k \to +\infty$, as $\Lip_{\gamma_{k+1}\gets\gamma_k} \to 1$ and $\liminf_\kkk \lambda_k\frac{1 - \eta_k\lambda_k}{\eta_k} >0$, then $z_k - T_{\theta_k,\gamma_k}z_k  \to 0$, which proves \ref{t:FPR-convergence-ii}. In particular, $(z_k)_\kkk$ and $(T_{\theta_k,\gamma_k}z_k)_\kkk$ have the same weak cluster points. Now suppose $(\theta,\gamma) \mapsto T_{\theta,\gamma}z$ is (strongly) continuous. In view of Corollary~\ref{l:parametric-DCP-extended} with $\mathcal{P} = \Theta \times \Gamma$,   any cluster point of $(z_k)_\kkk$ lies in $\Fix T_{\Bar{\theta},\Bar{\gamma}}$. We conclude that \ref{t:FPR-convergence-iii} holds from Lemma~\ref{l:Opial}\ref{p:Opial-cluster} and Lemma~\ref{l:aux-seq}.\end{proof}

\begin{remark}[Complexity estimates of the residual] In the context of the proof above, observe that if we assume the stricter condition $\sum_\kkk(\Lip_{\gamma_{k+1}\gets \gamma_k}^2 -1) < +\infty $, then from \begin{equation*}
		\|z_{k+1} - c_{k+1}\|^2 \leq \Lip_{\gamma_{k+1}\gets\gamma_k}^2\|z_k - c_k\|^2 - \Lip_{\gamma_{k+1}\gets\gamma_k}^2\lambda_k\frac{1 - \eta_k\lambda_k}{\eta_k}\|T_{\theta_k,\gamma_k}z_k - z_k\|^2
	\end{equation*} and using Lemma~\ref{f:RS} with $\alpha_k = \|z_k - c_k\|^2$, $ \varepsilon_k = \Lip_{\gamma_{k+1}\gets\gamma_k}^2 -1$ and $\beta_k = \Lip_{\gamma_{k+1}\gets\gamma_k}\eta_k^{-1}\lambda_k (1 - \eta_k\lambda_k)\|T_{\theta_k,\gamma_k}z_k - z_k\|^2$ for all $k \geq k_0$, then \begin{equation*}
		\sum_\kkk \Lip_{\gamma_{k+1}\gets\gamma_k}^2\lambda_k\frac{1 - \eta_k\lambda_k}{\eta_k}\|T_{\theta_k,\gamma_k}z_k - z_k\|^2 < +\infty.
	\end{equation*} Since $\Lip_{\gamma_{k+1}\gets\gamma_k} \geq 1$ and $\liminf_\kkk \lambda_k\frac{1 - \eta_k\lambda_k}{\eta_k} >0$, we obtain $\sum_\kkk \|T_{\theta_k,\gamma_k}z_k - z_k\|^2 < +\infty$. By virtue of \cite[Lemma 1(1)(a)]{davis2017convergence}, \begin{equation*}
		\min_{i=0,\dots,k}\|(\Id-T_{\theta_i,\gamma_i})z_i\|^2 = o\left(\frac{1}{k+1}\right).
	\end{equation*} Furthermore, from convexity of the squared norm and boundedness of  $(\sum_{i=0}^k\| T_{\theta_i,\gamma_i}z_i - z_i\|^2)_{\kkk}$,  we get an ergodic rate of convergence for the residuals \begin{equation*}
		\left\|\frac{1}{k+1}\sum_{i=0}^k (\Id-T_{\theta_i,\gamma_i})z_i\right\| = O\left(\frac{1}{\sqrt{k+1}}\right).
	\end{equation*}
	For a further analysis of the rates of convergence of relocated fixed-point iterations under the assumption of error bounds, we refer to \cite{atenas2025linear}.
\end{remark}

%%%%%%%%%%%%%%%%%%%%%%%%%%%%%%%%%%%%%%%%%%%%%%%%%%%%%%%%%%%%%%%%%%

\section{Variable stepsize  distributed forward-backward splitting methods} \label{s:FB}

Given $n \geq 2$ and $p \geq 1$, consider $n$ maximally monotone operators $A_1, \dots, A_n$ on $X$, and $p$ $\beta$-cocoercive operators $B_1, \dots, B_p$ on $X$. Our goal is to devise variable stepsize splitting methods to solve the  monotone inclusion problem in \eqref{e:distributed-monotone-inclusion}. In this section, we apply the relocated fixed-point iteration theory of Section~\ref{s:relocated-conic} to analyse relocated distributed splitting algorithms of forward-backward type. Following the framework introduced in \cite{dao2025general},  distributed forward-backward splitting algorithms are constructed using matrices that comply with the following assumptions.

\begin{assumption}\label{a:stand}
	The coefficient matrices $D = \diag(d_1,\dots, d_n)\in\mathbb{R}^{n\times n}$, $M \in \mathbb{R}^{n \times m}$, $N \in \mathbb{R}^{n \times n}$, $P\in\mathbb{R}^{n\times p}$, and $R \in \mathbb{R}^{p\times n}$ satisfy the following properties.
	\begin{enumerate}
		\item\label{a:stand_kerM}
		$\ker(M^*) = \mathbb{R}\bone$.
		\item\label{a:stand_P}
		$P^*\bone=\bone$.
		\item\label{a:stand_R}
		$R\bone=\bone$.
		\item\label{a:stand_neg_semidef}
		$2D - N - N^* - MM^*  \succeq 0$.
		\item\label{a:stand_N}
		$\sum_{i,j=1}^n N_{ij} = \sum_{i=1}^n d_i$.
		\item \label{a:stand_triangular} $N$ and $P$ are 
		lower triangular matrices with zeros on the diagonal, and $R$ is a lower triangular matrix.
	\end{enumerate}
\end{assumption}

Recall that for $A \in \{M,N, P, R\}$, we use the notation $\mathbf{A} = A \otimes \Id$. In this manner, the distributed forward-backward iteration operator that defines \cite[Algorithm 1]{dao2025general} (cf. \cite[Remark 3.1(iii)]{dao2025general}) is given by: for all $\theta,\gamma \in \R_{++}$, \begin{equation} \label{e:FB-operator} 
	\begin{aligned} T_{\theta, \gamma}: \; X^{m} &\to X^{m} \\
		\bz &\mapsto \bz - \theta \bM^* \bx^\gamma(\bz),
	\end{aligned} 
\end{equation} where the mapping $\bx^\gamma=(\bx_1^\gamma, \dots, \bx_n^\gamma): X^m \to X^n$ is defined, for all $\bz \in X^m$, as
\begin{equation} \label{e:FB-resolvent}
	\left\{\begin{aligned}
		\bx_1^\gamma(\bz) & := J_{\frac{\gamma}{d_1}A_1}\left(\frac{1}{d_1}(\bM\bz)_1\right)\\
		\bx_i^\gamma(\bz) &:= J_{\frac{\gamma}{d_i}A_i}\left( \frac{1}{d_i}(\bM\bz)_i + \frac{1}{d_i}(\bN\bx^\gamma(\bz))_{\le i-1} - \frac{\gamma}{d_i}\big(\bP \bB \bR \bx^\gamma(\bz)\big)_{\le \min\{i-1,p\}} \right) \text{~for~} i=2, \dots, n.
	\end{aligned}\right.
\end{equation} where for all $i = 1,\dots, n-1$, and $\bx \in X^n$, \begin{equation} \label{e:N}
	(\bN\bx)_{\le i-1} := \sum_{j=1}^{i-1} N_{ij}\bx_j,
\end{equation} and \begin{equation*}
	\begin{aligned}
		\big(\bP \bB \bR \bx\big)_{\le \min\{i-1,p\}} :=\sum_{j=1}^{\min\{i-1,p\}}P_{ij}B_j\left(\sum_{t=1}^{j} R_{jt}\bx_t\right).
	\end{aligned}
\end{equation*}

\begin{remark} \label{r:explicit} For ease of presentation, in this work we analyse the cocoercive case. Similar results can be obtained when cocoercivity is replaced with Lipschitz continuity (see \cite{dao2025general}). Furthermore, note that $\theta\in\R_{++}$ in \eqref{e:FB-operator} can be understood as a relaxation parameter, since for all $\bz \in X^m$, $T_{\theta,\gamma}\bz= (1-\theta)\bz + \theta (\bz - \bM^* \bx^\gamma(\bz))$. Therefore, to simplify the analysis of relocated fixed-point iterations, for all $\kkk$, we take  $\lambda_k = 1$ hereafter.
\end{remark}

\begin{lemma}[{\cite[Lemma 3.6]{dao2025general}}] \label{l:con-ave} 
	Let $A_1, \dots, A_n$ be  maximally monotone operators on $X$, and $B_1, \dots, B_{p}$ be  $\beta$-cocoercive operators on $X$. Suppose Assumption~\ref{a:stand}\ref{a:stand_kerM}--\ref{a:stand_N} hold. Then, for all $\theta, \gamma \in \R_{++}$, $T_{\theta,\gamma}$ is conically $\eta(\theta,\gamma)$-averaged with $ \eta(\theta,\gamma) = \frac{2\theta}{2 - \gamma\mu}$, when $0 < \gamma < \frac{2}{\mu}$ and $\mu = \beta\|(P^* -R)(M^*)^\dagger\|^2$. The convention for this result is that $\frac{1}{\mu} = +\infty$ if $\mu = 0$.
\end{lemma}

In \cite{atenas2025relocated}, the variable stepsize variant of a family of resolvent splitting methods is constructed as  relocated fixed-point iterations. The first step to build a similar construction for distributed forward-backward methods is to define an appropriate fixed-point relocator. We define the \Qname{} of the iteration operator in \eqref{e:FB-operator} based on the following mapping. Given $\gamma \in \R_{++}$, define  $\be^\gamma: X^m  \to X^n$ by: for all $\bz \in X^m$, \begin{equation} \label{eq:def-e}
	\begin{aligned}
		\be^\gamma(\bz) &:= P_{\range(\bM)}\big(d_i \bx^\gamma_i(\bz) - (\bN \bx^\gamma(\bz))_{\leq i-1}\big)_{i=1}^n 
	\end{aligned}
\end{equation} where $\bN \bx^\gamma(\bz)$ is obtained by using \eqref{e:FB-resolvent} and \eqref{e:N}, and from Assumption~\ref{a:stand}\ref{a:stand_kerM}, \begin{equation} \label{e:range-M}
	\range(\bM) = \left\{ \by \in X^n: \sum_{i=1}^n \by_i = 0\right\}.
\end{equation} The next result presents fundamental properties satisfied by the operator $\be^\gamma$ in \eqref{eq:def-e} over the set of fixed points of the iteration operator in \eqref{e:FB-operator} that will yield the satisfaction of Definition~\ref{d:FPR}\ref{a:trans_bijection}-\ref{a:trans_semigroup} for the fixed-point relocator for \eqref{e:FB-operator}.

\begin{proposition} \label{l:known-e}
	
	Let $A_1, \dots, A_n$ be set-valued operators on $X$, and $B_1, \dots, B_{p}$ be single-valued operators on $X$. For $\theta ,\gamma \in \R_{++}$, consider the operator $T_{\theta,\gamma}$ defined in \eqref{e:FB-operator} and $\bx^\gamma = (\bx^\gamma_1, \dots, \bx^\gamma_n)$ defined in \eqref{e:FB-resolvent}. Then, the following hold. 
	\begin{enumerate} 
		
		\item\label{l:known-e-ii} If Assumption~\ref{a:stand}\ref{a:stand_kerM} holds, then for all $\bz \in \Fix T_{\theta,\gamma}$, there exists $\overline{x} \in X$ such that for all $i = 1, \dots, n$, $\bx_i^\gamma(\bz) = \overline{x}$.

		\item \label{l:known-e-iii} If Assumption~\ref{a:stand}\ref{a:stand_kerM},~\ref{a:stand_triangular} and~\ref{a:stand_N} hold, then for all $\bz \in \Fix T_{\theta,\gamma}$, and $\Bar{x}$ from \ref{l:known-e-ii}, $$\be^\gamma(\bz)  = \left(d_1 \overline{x} , \dots, \left(d_i - \sum_{j=1}^{i-1} N_{ij}\right) \overline{x} , \dots ,\left(d_n - \sum_{j=1}^{n-1}N_{nj}\right) \overline{x}\right).$$

		\item \label{l:known-e-0}   If Assumption~\ref{a:stand}\ref{a:stand_kerM},~\ref{a:stand_P},~\ref{a:stand_R}, and~\ref{a:stand_N} hold, then: \begin{equation*}
			\Fix T_{\theta,\gamma} \neq \emptyset \text{~if and only if~} \zer\left(\sum_{i=1}^n A_i + \sum_{i=1}^p B_i\right) \neq \emptyset.
		\end{equation*}
	\end{enumerate}
	
\end{proposition}

\begin{proof}
	Let $\theta,\gamma \in \R_{++}$ and $\bz \in \Fix T_{\theta,\gamma}$. From \eqref{e:FB-operator}, it follows that $\bx^\gamma(\bz) \in \ker(\bM^*)$, and thus Assumption~\ref{a:stand}\ref{a:stand_kerM} implies that there exists $\Bar{x} \in X$ such that for all $i = 1, \dots, n$, $\bx^\gamma_i(\bz) = \Bar{x}$, proving \ref{l:known-e-ii}. Furthermore, to prove \ref{l:known-e-iii}, from \eqref{eq:def-e} and \eqref{e:range-M}, there holds \begin{equation} \label{eq:e-interm}
		\begin{aligned}
			\be^\gamma_i(\bz) &= d_i \bx^\gamma_i(\bz) - (\bN \bx^\gamma(\bz))_{\le i-1} - \frac{1}{n}\sum_{j=1}^n \Bigl(d_j \bx^\gamma_j(\bz) - (\bN \bx^\gamma(\bz))_{\le j-1}\Bigr)\\
			& = d_i \Bar{x}- (\bN \bx^\gamma(\bz))_{\le i-1} - \frac{1}{n}\sum_{j=1}^n \Bigl(d_j \Bar{x} - (\bN \bx^\gamma(\bz))_{\le j-1}\Bigr).
		\end{aligned}
	\end{equation} Observe that it suffices to show that $\sum_{j=1}^n \Bigl(d_j \Bar{x} - (\bN \bx^\gamma(\bz))_{\le j-1}\Bigr) = 0$. Indeed, for $j=1$, $(\bN \bx^\gamma(\bz))_{\le 0} = 0$. For $j=2,\dots, n$,   \begin{equation} \label{e:Nx-at-fix}
		(\bN \bx^\gamma(\bz))_{\le j- 1} = \sum_{t=1}^{j-1} N_{jt}\bx^\gamma_t(\bz) = \sum_{t=1}^{j-1} N_{jt}\Bar{x},
	\end{equation} and thus in view of Assumption~\ref{a:stand}\ref{a:stand_triangular} and~\ref{a:stand_N}, \begin{equation*}
		\sum_{j=1}^n (\bN \bx^\gamma(\bz))_{\le j- 1} = \sum_{j=1}^n\sum_{t=1}^{j-1} N_{jt}\Bar{x} =  \sum_{j=1}^n\sum_{t=1}^{n} N_{jt}\Bar{x} = \sum_{j=1}^n d_j\Bar{x},
	\end{equation*} from where  $\sum_{j=1}^n \Bigl(d_j \Bar{x} - (\bN \bx^\gamma(\bz))_{\le j-1}\Bigr) = 0$ follows. Altogether, we get for $i=2, \dots, n$ that \begin{equation*}
		\be^\gamma_i(\bz) = d_i \Bar{x}- \sum_{j=1}^{i-1} N_{ij}\Bar{x}.
	\end{equation*} Observe that for $i=1$, $\mathbf{e}_1^\gamma(\mathbf{z}) = d_1\bar{x}$ as $(\bN \bx^\gamma(\bz))_{\le 0} = 0$. Hence, \ref{l:known-e-iii} follows. Finally, item \ref{l:known-e-0} is \cite[Lemma 3.5]{dao2025general}.\end{proof}

The definition of a \Qname{} also requires the Lipschitz continuity condition in Definition~\ref{d:FPR}\ref{a:trans_nonex-type}. In the following result, we show Lipschitz continuity of the mapping $\be^{\gamma}$ in \eqref{eq:def-e}, which will imply Definition~\ref{d:FPR}\ref{a:trans_nonex-type} for the fixed-point relocator for \eqref{e:FB-operator}.

\begin{proposition} \label{l:known-e-i} 
	Let $A_1, \dots, A_n$ be  maximally monotone operators on $X$, and $B_1, \dots, B_{p}$ be  $\beta$-cocoercive operators on $X$. Suppose Assumption~\ref{a:stand}\ref{a:stand_triangular} holds. Then, the operator $\be^\gamma$ defined in \eqref{eq:def-e} is Lipschitz continuous with constant \begin{equation*}
		\sqrt{C_1^2 + \sum_{i=2}^n \left(C_i(\gamma)  + \|\bN\| \sum_{j=1}^{i-1} \frac{C_j(\gamma)}{d_j}  \right)^2},
	\end{equation*}where \begin{equation} \label{e:C-lip}
		\left\{\begin{aligned}
			C_1  & := \sqrt{m}\|\bM\| \\
			C_i(\gamma)  & := \sqrt{m}\|\bM \| + \|\bN\| \sum_{j=1}^{i-1} \frac{C_j(\gamma)}{d_j} + \gamma\beta\| \bP\| \| \bR\| \sum_{j=1}^{\min\{i-1,p\}}\sum_{t=1}^{j} \frac{C_t(\gamma)}{d_t}
			\text{~for~} i = 2, \dots, n,
		\end{aligned}\right.
	\end{equation} using $C_1(\gamma) := C_1$ for ease of notation.
\end{proposition}

\begin{proof}
	In view of nonexpansiveness of $P_{\range(\bM)}$,  we only need to bound, for all $\bz,\bw \in X^m$,  \begin{align*}
		&\|\bx^\gamma_1(\bz) - \bx^\gamma_1(\bw)\|\text{~and~} \\ &\| d_i(\bx^\gamma_i(\bz)-\bx^\gamma_i(\bw))- ( \bN [\bx^\gamma(\bz)-\bx^\gamma(\bw)])_{\le i-1}\| \text{~for~} i = 2, \dots, n.
	\end{align*}   
	First, observe that for all $i = 1,\dots,n $, $j = 1,\dots, m$, \begin{equation} \label{e:norm-ineq} \text{~for~} A \in \{M,N, P, P^*, R\}, \;  |A_{ij}| \le \|\bA\|.   \end{equation} We show that for every $i=1,\dots, n$ and $\gamma \in \R_{++}$, $\bx^\gamma_i$ is Lipschitz continuous with constant $C_i(\gamma)/d_i$. For $i=1$, \eqref{e:FB-resolvent},  nonexpansiveness of the resolvent, and \eqref{e:norm-ineq} yield \begin{equation*} 
    \|\bx^\gamma_1(\bz) - \bx^\gamma_1(\bw)\| \leq \frac{1}{d_1}\sum_{j=1}^{m}|M_{1j}|\|\bz_j - \bw_j\|\leq \frac{\sqrt{m}\|\bM\|}{d_1}\|\bz-\bw\| = \dfrac{C_1}{d_1}\|\bz-\bw\|. \end{equation*}  Next, given $i = 2, \dots, n$, suppose for all $j = 1, \dots, i-1$, $\|\bx^\gamma_j(\bz) - \bx^\gamma_j(\bw)\| \le \frac{C_j(\gamma)}{d_j} \|\bz-\bw\| $. Then, from \eqref{e:FB-resolvent}, and the fact that $B_j$ is Lipschitz continuous with constant $\beta$, \begin{align*}
		\|\bx^\gamma_i(\bz) - \bx^\gamma_i(\bw)\| & \leq \frac{1}{d_i}\left(\sum_{j=1}^m|M_{ij}|\|\bz_j - \bw_j\| + \sum_{j=1}^{i-1} |N_{ij}| \| \bx^\gamma_j(\bz) - \bx^\gamma_j(\bw)\| \right. \\ &\left. \quad\quad\quad\quad + \gamma \sum_{j=1}^{\min\{i-1,p\}} |P_{ij}| \beta \sum_{t=1}^{j}|R_{jt}| \| \bx^\gamma_t(\bz) - \bx^\gamma_t(\bw) \| \right) \\ 
		& \leq \frac{1}{d_i} \left( \sqrt{m}\|\bM \| + \|\bN\| \sum_{j=1}^{i-1} \frac{C_j(\gamma)}{d_j} + \gamma\beta\| \bP\| \| \bR\| \sum_{j=1}^{\min\{i-1,p\}}\sum_{t=1}^{j} \frac{C_t(\gamma)}{d_t}  \right) \|\bz - \bw\|.
	\end{align*}  Hence, the claim for $\bx_i^\gamma$ follows. In this manner, for $i=2,\dots, n$, \begin{align*}  \| d_i(\bx^\gamma_i(\bz)-\bx^\gamma_i(\bw))- ( \bN [\bx^\gamma(\bz)-\bx^\gamma(\bw)])_{\le i-1}\| 
	& \leq  \;  d_i\|\bx^\gamma_i(\bz)-\bx^\gamma_i(\bw) \| + \sum_{j=1}^{i-1} |N_{ij}|\|\bx^\gamma_j(\bz)-\bx^\gamma_j(\bw)\|  \\%
	& \leq \; \left(C_i(\gamma)  + \|\bN\| \sum_{j=1}^{i-1} \frac{C_j(\gamma)}{d_j}  \right) \|\bz - \bw\|. \end{align*}Finally,
\begin{equation*}
	\begin{aligned}
		 \|\be^\gamma(\bz) - \be^\gamma(\bw)\|  
		& \leq  \; \sqrt{d_1^2\|\bx^\gamma_1(\bz)-\bx^\gamma_1(\bw)\|^2 + \sum_{i=2}^n \| d_i(\bx^\gamma_i(\bz)-\bx^\gamma_i(\bw))- ( \bN [\bx^\gamma(\bz)-\bx^\gamma(\bw)])_{\le i-1}\|^2}\\
		& \leq  \;\sqrt{C_1^2 + \sum_{i=2}^n \left(C_i(\gamma)  + \|\bN\| \sum_{j=1}^{i-1} \frac{C_j(\gamma)}{d_j}  \right)^2} \|\bz-\bw\|,
	\end{aligned}
\end{equation*} which concludes the proof.\end{proof}

In the next result, building on the properties of $\be^\gamma$, we propose a \Qname{} for the distributed forward-backward operator (cf. \eqref{e:FPR-DR}).
		
\begin{proposition}[Fixed-point relocator for distributed forward-backward methods]\label{p:FPR-Q}
	Let $A_1, \dots, A_n$ be  maximally monotone operators on $X$, and $B_1, \dots, B_{p}$ be  $\beta$-cocoercive operators on $X$. For $\theta ,\gamma \in \R_{++}$, consider the operator $T_{\theta,\gamma}$ defined in \eqref{e:FB-operator} and $\bx^\gamma = (\bx^\gamma_1, \dots, \bx^\gamma_n)$ defined in \eqref{e:FB-resolvent}. Suppose Assumption~\ref{a:stand}\ref{a:stand_kerM},~\ref{a:stand_R},~\ref{a:stand_N} and~\ref{a:stand_triangular} hold true. Let $\gamma, \delta \in \R_{++}$, and define the operator $Q_{\delta\gets \gamma}: X^{m} \to X^{m}$ for all $\bz \in X^{m}$ by \begin{equation} \label{e:FPR-Q}
		Q_{\delta\gets \gamma}\bz := \tfrac{\delta}{\gamma}\bz +  \left(1 - \tfrac{\delta}{\gamma} \right)\bM^{\dagger}\be^\gamma(\bz),
	\end{equation} where  $\be^\gamma$ is given in \eqref{eq:def-e}, and $\bM^{\dagger}$ denotes the pseudo-inverse of $\bM = M \otimes I$.   
	Then 
	$(Q_{\delta\gets \gamma})_{\gamma, \delta \in \R_{++}}$ 
	define  \Qname{s} for the distributed forward-backward operators $(T_{\theta,\gamma})_{\gamma\in\R_{++}}$ given by \eqref{e:FB-operator}-\eqref{e:FB-resolvent}, with  Lipschitz constants\begin{equation} \label{e:Lip-const}
		\Lip_{\delta\gets\gamma} = \max\left\{1,\frac{\delta}{\gamma} +  \left|1 - \frac{\delta}{\gamma} \right|\|\bM^{\dagger}\|\sqrt{C_1^2 + \sum_{i=2}^n \left(C_i(\gamma)  + \|\bN\| \sum_{j=1}^{i-1} \frac{C_j(\gamma)}{d_j}  \right)^2}\right\},
	\end{equation} and $C_1$ and $C_2(\gamma), \dots, C_n(\gamma)$ given in  \eqref{e:C-lip}. 
\end{proposition}	
\begin{proof}
	Since $\range(\bM)$ is closed and $\be^\gamma(\bz) \in \range(\bM)$, in view of \cite[Proposition 3.27]{BC2017},  for all $\delta, \gamma \in \R_{++}$ and $\bz \in X^{m}$, the system \begin{equation}\label{e:unique-sol}
		\bM\by = \tfrac{\delta}{\gamma}\bM\bz +  \left(1 - \tfrac{\delta}{\gamma} \right)\be^\gamma(\bz)
	\end{equation} has at least one solution. Uniqueness of the solution follows from injectivity of $\bM$. Hence, the mapping of unique solutions to this system is exactly $Q_{\delta\gets\gamma}$. 
	
	In order to show that $Q_{\delta\gets\gamma}$ is 
	a \Qname{}, we first prove that \begin{equation} \label{e:QFix-inc-Fix}
		Q_{\delta\gets \gamma}\Fix T_{\theta,\gamma} \subseteq \Fix T_{\theta,\delta}.
	\end{equation} Let $\bz \in \Fix T_{\theta,\gamma}$  
	and $\overline{x} \in X$ from Proposition~\ref{l:known-e}\ref{l:known-e-ii} such that $ \bx_i^\gamma(\bz) = \overline{x}$ for all $ i = 1, \dots, n$, and set $\by = Q_{\delta\gets \gamma}\bz$. Our goal, in view of \eqref{e:FB-operator}, is to prove that $\bM^*\bx^\delta(\by) = 0$. From \eqref{e:unique-sol} and Proposition~\ref{l:known-e}\ref{l:known-e-iii}, \begin{equation} \label{e:Q-at-fixed-2}
		\bM \by =   \frac{\delta}{\gamma}\bM \bz + \left(1-\frac{\delta}{\gamma}\right)\left(d_1 \overline{x} , \dots, \left(d_i - \sum_{j=1}^{i-1} N_{ij}\right) \overline{x} , \dots ,\left(d_n - \sum_{j=1}^{n-1}N_{nj}\right) \overline{x}\right).
	\end{equation} First, for $i=1$, \begin{align*}
		x_1^\delta(\by) &= J_{\frac{\delta}{d_1}A_1}\left(\frac{1}{d_1}(\bM \by)_1\right) = J_{\frac{\delta}{d_1}A_1}\left(\frac{\delta}{\gamma}\frac{1}{d_1}(\bM\bz)_1 + \left(1-\frac{\delta}{\gamma}\right) \overline{x}\right) \\
		&=  J_{\frac{\delta}{d_1}A_1}\left(\frac{\delta}{\gamma}\frac{1}{d_1}(\bM\bz)_1 + \left(1-\frac{\delta}{\gamma}\right) \bx_1^\gamma(\bz)\right) \\
		&=  J_{\frac{\delta}{d_1}A_1}\left(\frac{\delta}{\gamma}\frac{1}{d_1}(\bM\bz)_1 + \left(1-\frac{\delta}{\gamma}\right) J_{\frac{\gamma}{d_1}A_1}\left(\frac{1}{d_1}(\bM \bz)_1\right)\right) \\
		& = J_{\frac{\gamma}{d_1}A_1}\left(\frac{1}{d_1}(\bM \bz)_1\right) = x_1^\gamma(\bz) =  \overline{x},
	\end{align*} where in the second equality of the first line we use \eqref{e:Q-at-fixed-2}, in the third line we use \eqref{e:FB-resolvent}, and in the first equality of the last line we use Lemma~\ref{l:known-J}\ref{l:known-J-1}. Next, let $i = 2, \dots, n$, and suppose for all $j = 1, \dots, i -1$, $\bx_j^\delta(\by) = \overline{x}$. From Assumption~\ref{a:stand}\ref{a:stand_R} and~\ref{a:stand_triangular}, by denoting $\Bar{\bx} = (\Bar{x}, \dots, \Bar{x}) \in X^n$, it follows that $\bR \Bar{\bx} = \Bar{\bx}$, and $\sum_{t=1}^{j} R_{jt}\Bar{x} = \Bar{x}$. In this manner,  similarly to  \eqref{e:Nx-at-fix}, it holds that\begin{equation} \label{e:aux-eq}
		(\bN \bx^\delta(\by))_{\leq i-1} =\sum_{j=1}^{i-1} N_{ij}\Bar{x}, \text{~and~}
		\left\{\begin{aligned}
			&\text{for~} \bF_i \bar{x}:=\sum_{j=1}^{\min\{i-1,p\}}P_{ij} B_j\Bar{x}\\ 
			&\big( \bP  \bB \bR \bx^\delta(\by) \big)_{\le \min\{i-1,p\}} = \bF_i \bar{x} = \big( \bP  \bB \bR \bx^\gamma(\bz) \big)_{\le \min\{i-1,p\}}. 
		\end{aligned}\right.\end{equation} 
	Hence, \begin{align*}
		\bx_i^\delta(\by) &= J_{\frac{\delta}{d_i}A_i}\left(\frac{1}{d_i}(\bM\by)_i + \frac{1}{d_i}\sum_{j=1}^{i-1} N_{ij}\Bar{x} - \frac{\delta}{d_i} \bF_i \bar{x} \right) \\
		& = J_{\frac{\delta}{d_i}A_i}\left(\frac{1}{d_i}\frac{\delta}{\gamma}(\bM\bz)_i + \frac{1}{d_i}\left(1 - \frac{\delta}{\gamma}\right)\left(d_i - \sum_{j=1}^{i-1}N_{ij}\right)\Bar{x}  + \frac{1}{d_i}\sum_{j=1}^{i-1} N_{ij}\Bar{x} -\frac{\delta}{d_i} \bF_i \Bar{x} \right) \\
		& = J_{\frac{\delta}{d_i}A_i}\left(\frac{1}{d_i}\frac{\delta}{\gamma}(\bM\bz)_i + \left(1 - \frac{\delta}{\gamma}\right)\Bar{x} + \frac{1}{d_i}\frac{\delta}{\gamma} \sum_{j=1}^{i-1}N_{ij}\Bar{x} -\frac{\delta}{d_i} \bF_i \Bar{x} \right)  \\ 
        & = J_{\frac{\delta}{d_i}A_i}\left(\frac{\delta}{\gamma} \left[ \frac{1}{d_i}(\bM\bz)_i +\frac{1}{d_i} \sum_{j=1}^{i-1}N_{ij}\Bar{x}  - \frac{\gamma}{d_i} \bF_i \Bar{\bx} \right] + \left(1 - \frac{\delta}{\gamma}\right)\Bar{x}\right)  \\ 
        & = J_{\frac{\delta}{d_i}A_i}\left(\frac{\delta}{\gamma} \left[ \frac{1}{d_i}(\bM\bz)_i +\frac{1}{d_i} \big( \bN \bx^\gamma(\bz) \big)_{\leq i-1}
		- \frac{\gamma}{d_i} \big( \bP  \bB \bR \bx^\gamma(\bz) \big)_{\le \min\{i-1,p\}} \right] + \left(1 - \frac{\delta}{\gamma}\right)\bx^\gamma_i(\bz)\right)   \\ 
        & = \bx^\gamma_i(\bz) = \Bar{x},
	\end{align*} where in the first equality we use \eqref{e:FB-resolvent}, the second equality follows from \eqref{e:Q-at-fixed-2}, in the fifth equality we use \eqref{e:aux-eq}, and in the last line we employ Lemma~\ref{l:known-J}\ref{l:known-J-1}. Altogether, we conclude that for all $i = 1, \dots, n$, $\bx_i^\delta(\by) = \overline{x}$, thus $\bx^\delta(\by) \in \ker(\bM^*)$ and $\by \in \Fix T_{\theta,\delta}$. Hence \eqref{e:QFix-inc-Fix} holds. In order to prove the reverse inclusion, we first show that Definition~\ref{d:FPR}\ref{a:trans_semigroup} holds. In other words, we show that for all $\theta \in \R_{++}$, $\gamma,\delta,\varepsilon \in \R_{++}$, \begin{equation} \label{3-group-prop}
		\text{for all~} \bz \in \Fix T_{\theta,\gamma},\;Q_{\varepsilon\gets\delta}Q_{\delta\gets\gamma}\bz=Q_{\varepsilon\gets\gamma}\bz.
	\end{equation}  We argue similarly to \cite[Eq. (5.16)]{atenas2025relocated}. Let $\bz \in \Fix T_{\theta,\gamma}$ and $\by = Q_{\delta\gets \gamma}\bz$. Then, from the analysis above, for all $i = 1, \dots, n$, $\bx_i^\delta(\by) = \overline{x} = \bx_i^\gamma(\bz)$, and thus $\be^\delta(\by) = \be^\gamma(\bz)$. Then\begin{equation*}\begin{aligned}     Q_{\varepsilon\gets \delta}\by & = \tfrac{\varepsilon}{\delta}\by + \left(1-\tfrac{\varepsilon}{\delta}\right)\bM^{\dagger}\be^\delta(\by)\\ %
			& = \tfrac{\varepsilon}{\delta}\by + \left(1-\tfrac{\varepsilon}{\delta}\right)\bM^{\dagger}\be^\gamma(\bz)\\ 
			& = \tfrac{\varepsilon}{\delta}\left[ \tfrac{\delta}{\gamma}\bz + \left( 1 - \tfrac{\delta}{\gamma}\right)\bM^{\dagger}\be^\gamma(\bz) 
			\right] + \left(1-\tfrac{\varepsilon}{\delta}\right)\bM^{\dagger} \be^\gamma(\bz)\\ 
			& =  \tfrac{\varepsilon}{\gamma}\bz + \left( 1 - \tfrac{\varepsilon}{\gamma}\right)\bM^{\dagger}\be^\gamma(\bz)  = Q_{\varepsilon\gets \gamma}\bz,
	\end{aligned}\end{equation*} where in the first equality we use \eqref{e:FPR-Q}, and in the third equality we employ \eqref{e:Q-at-fixed-2}. This  proves \eqref{3-group-prop}, which is  Definition~\ref{d:FPR}\ref{a:trans_semigroup}.

	Next, for all $\delta,\gamma \in \R_{++}$, from \eqref{3-group-prop} with $\varepsilon = \gamma$, then $Q_{\gamma\gets\delta}Q_{\delta\gets\gamma} = Q_{\gamma\gets\gamma} = I$ on $\Fix T_{\theta,\gamma}$. Likewise, the relation in \eqref{3-group-prop} also implies $Q_{\delta\gets\gamma}Q_{\gamma\gets\delta} = Q_{\delta\gets\delta} = I$ on $\Fix T_{\theta,\delta}$. Hence 
	${Q_{\delta\gets\gamma}}_{|\Fix T_{\theta,\gamma}}$ is a bijection with inverse ${Q_{\delta\gets\gamma}}_{|\Fix T_{\theta,\gamma}}^{-1} = {Q_{\gamma\gets\delta}}_{|\Fix T_{\theta,\delta}}$. Moreover, from \eqref{e:QFix-inc-Fix}, it follows that $ Q_{\gamma\gets \delta}\Fix T_{\theta,\delta} \subseteq \Fix T_{\theta,\gamma}$, then $\Fix T_{\theta,\delta} \subseteq Q_{\delta \gets \gamma} \Fix T_{\theta,\gamma}$. Therefore, \eqref{e:QFix-inc-Fix} holds with equality, and thus ${Q_{\delta\gets\gamma}}_{|\Fix T_{\theta,\gamma}}$ is a bijection that maps $\Fix T_{\theta,\gamma}$ to $\Fix T_{\theta,\delta}$, proving Definition~\ref{d:FPR}\ref{a:trans_bijection}.
	
	Moreover, for every $\gamma \in \R_{++}$ and $\bz \in X^{m}$, the mapping $\Gamma\to X\colon \delta\mapsto Q_{\delta\gets \gamma}\bz$ is continuous by construction, as it is an affine mapping, so that Definition~\ref{d:FPR}\ref{a:trans_cont} holds.
	
	Finally, we show Definition~\ref{d:FPR}\ref{a:trans_nonex-type}. For any $\delta,\gamma\in\R_{++}$, and $\bz,  \bw \in X^{m}$, from the definition of $Q_{\delta\gets\gamma}$, it holds that\begin{equation*}
		\begin{aligned}
			\|Q_{\delta\gets \gamma}\bz - Q_{\delta\gets \gamma}\bw\| & \le  \frac{\delta}{\gamma}\|\bz-\bw\| +  \left|1 - \frac{\delta}{\gamma} \right|\|\bM^{\dagger}\|\|\be^\gamma(\bz) - \be^\gamma(\bw)\| 
		\end{aligned} 
	\end{equation*} Hence, from Proposition~\ref{l:known-e-i}, $Q_{\delta\gets\gamma}$ is Lipschitz continuous with a Lipschitz constant given by \eqref{e:Lip-const}. We therefore conclude that $Q_{\delta\gets\gamma}$ is a fixed-point relocator for $T_{\theta,\gamma}$.\end{proof}

\begin{remark}
    Since $P_{\range(\bM)} = \bM\bM^\dagger$ \cite[Proposition 3.30(ii)]{BC2017} and $\bM^{\dagger}\bM\bM^{\dagger} = \bM^{\dagger}$, it follows that \[
\bM^{\dagger} \be^{\gamma}(\bz)
= \bM^{\dagger}P_{\range(\bM)}\big(d_i \bx^\gamma_i(\bz) - (\bN \bx^\gamma(\bz))_{\leq i-1}\big)_{i=1}^n = \bM^{\dagger}\big(d_i \bx^\gamma_i(\bz) - (\bN \bx^\gamma(\bz))_{\leq i-1}\big)_{i=1}^n.
\] Hence, $$Q_{\delta\gets \gamma}\bz = \tfrac{\delta}{\gamma}\bz +  \left(1 - \tfrac{\delta}{\gamma} \right)\bM^{\dagger}(D-N) \otimes \bx^\gamma(\bz)$$ holds for all $\bz\in X^m$ for the \Qname{} in \eqref{e:FPR-Q}, and the analysis can be thus performed similarly with this alternative, equivalent form of the \Qname{}.
\end{remark}

Having defined a \Qname{} for the iteration operator in \eqref{e:FB-operator}, we now proceed to state the convergence result for the family of relocated distributed forward-backward methods.

\begin{theorem}[Convergence of relocated distributed forward-backward methods] \label{t:dist-FB-convergence}
	Let $A_1, \dots, A_n$ be  maximally monotone operators on $X$, and $B_1, \dots, B_{p}$ be  $\beta$-cocoercive operators on $X$, such that $\zer(\sum_{i=1}^n A_i + \sum_{i=1}^p B_i) \neq \varnothing$. Suppose Assumption~\ref{a:stand} holds. For $\theta,\gamma \in \R_{++}$,  consider the  operator $T_{\theta,\gamma}$ defined in \eqref{e:FB-operator}-\eqref{e:FB-resolvent}. Moreover, let  $(\theta_k)_\kkk $ in $\R_{++}$ be a sequence bounded and separated from zero, and $\Gamma \subseteq (0 , 2\mu^{-1})$ for $\mu = \beta\|(P^* -R)(M^*)^\dagger\|^2$. Furthermore, let $(\gamma_k)_\kkk $ in $\Gamma$ be a convergent sequence to some $\Bar{\gamma} \in \Gamma$, such that $\sum_\kkk (\gamma_{k+1} - \gamma_k)_{+} < +\infty$ and  $\liminf_\kkk(2- \gamma_k\mu - 2\theta_k) >0$. Let $(Q_{\delta\gets\gamma})_{\delta,\gamma \in \Gamma}$ be the \Qname{s} defined in  \eqref{e:FPR-Q}. 
	
	Given $\bz_0 \in X^{m}$, define for all $\kkk$, 
	\begin{equation} \label{e:var-FB}
		\left\{\begin{aligned}
			\bx_k & = \bx^{\gamma_k}(\bz_k) \text{~from \eqref{e:FB-resolvent}}\\
			\bw_k& = \bz_k - \theta_k \bM^*\bx_k\\
			\bz_{k+1}& = Q_{\gamma_{k+1}\gets\gamma_k}\bw_k.
		\end{aligned}\right.
	\end{equation} Then, $(\bz_k)_\kkk$ and $(\bw_k)_\kkk$ converge weakly to some $\overline{\bz} \in \Fix T_{\Bar{\theta},\overline{\gamma}}$ for any $\Bar{\theta} \in \Theta$, and  $ (\bx_{k})_\kkk$ converges weakly to  $\overline{\bx} = (\overline{x}, \dots, \overline{x}) \in X^n$, such that \begin{equation} \label{e:x-zero}
		\Bar{x} = J_{\frac{\overline{\gamma}}{d_1}A_1}\left(\frac{1}{d_1}(\bM\overline{\bz})_1\right)\in \zer\left(\sum_{i=1}^n A_i + \sum_{i=1}^{p} B_i \right).
	\end{equation}
\end{theorem}

The proof of Theorem~\ref{t:dist-FB-convergence} can be found in Appendix~\ref{a:convergence}. Barring the details related to \Qname{s}, this proof is an adaptation of the convergence analysis in \cite{dao2025general} tailored to the fixed-point relocator framework.  In particular, by taking $Q_{\delta\gets\gamma} = I$ for all $\delta,\gamma\in\R_{++}$, \eqref{e:var-FB} recovers \cite[Algorithm 1]{dao2025general} (see \cite[Remark
3.1(i)-(iii)]{dao2025general}), while Theorem~\ref{t:dist-FB-convergence} generalises \cite[Theorem 3.10]{dao2025general}. 

\begin{remark}[On the \Qname{} Lipschitz constants] \label{r:Lip-const}
	The condition $\sum_\kkk (\gamma_{k+1} - \gamma_k)_{+} < +\infty$ above relates to the condition $\sum_\kkk (\Lip_{\gamma_{k+1}\gets \gamma_k} -1) < +\infty$ for the Lipschitz constants of the \Qname{s} in Theorem~\ref{t:FPR-convergence}. Indeed, since $(\gamma_k)_\kkk$ is bounded, then there exist $\gamma_{\rm low}, \gamma_{\rm high} \in \R_{++}$ such that for all $\kkk$, $\gamma_{\rm low} \leq \gamma_k \leq \gamma_{\rm high}$. From \eqref{e:C-lip},  $\gamma \mapsto C_i(\gamma)$ are continuous mappings on $\R_{++}$ for all $i = 1, \dots, n$, thus all $ C_i([\gamma_{\rm low}, \gamma_{\rm high}])$ are uniformly bounded. Hence, from \eqref{e:Lip-const}, there exists $\Bar{C} \in \R_{++}$ such that for all $\kkk$, \begin{equation*}
		\Lip_{\gamma_{k+1}\gets\gamma_k} -1 \leq \left|\frac{\gamma_{k+1}}{\gamma_k} -1\right| +  \left|1 - \frac{\gamma_{k+1}}{\gamma_k}\right|  \|\bM^{\dagger}\|\Bar{C}   \leq \frac{|\gamma_{k+1}-\gamma_k|}{\gamma_{\rm low}}(1 +   \|\bM^{\dagger}\|\Bar{C}).
	\end{equation*} From \cite[Remark 4.10 (iii)]{atenas2025relocated}, having $\gamma_k \to \Bar{\gamma}$ such that $\sum_\kkk (\gamma_{k+1} - \gamma_k)_+ < +\infty$ is equivalent to having $(\gamma_k)_\kkk$ bounded and $\sum_\kkk |\gamma_{k+1} - \gamma_k| < +\infty$. Hence, $\sum_\kkk (\gamma_{k+1} - \gamma_k)_{+} < +\infty$ implies  $\sum_\kkk (\Lip_{\gamma_{k+1}\gets\gamma_k} -1) < +\infty$. 
\end{remark}

\begin{remark}\label{r:cheap-reloc}
	In practice, using the \Qname{} in \eqref{e:FPR-Q} may require extra resolvent evaluations, as argued in \cite[Section 5.2]{atenas2025relocated}, which increases the per-iteration cost of the resulting algorithm. More specifically, given $\bw \in X^{m}$, $\delta,\gamma \in \R_{++}$, for each $i=1,\dots, n$, the $i$-th component of the \Qname{} in \eqref{e:FPR-Q} evaluated at $\bz = \bw$ is determined by the term \begin{equation*}
		\begin{aligned}
		\big(\bM^{\dagger}\be^\gamma(\bw)\big)_i &= \sum_{j=1}^n \bM^{\dagger}_{ij}\be^\gamma_j(\bw) \\
			& = \sum_{j=1}^n \bM^{\dagger}_{ij}\left(d_j \bx^\gamma_j(\bw) - (\bN \bx^\gamma(\bw))_{\le j-1} - \frac{1}{n}\sum_{\ell=1}^n \Bigl(d_\ell \bx^\gamma_\ell(\bw) - (\bN \bx^\gamma(\bw))_{\le \ell-1}\Bigr) \right)\\
			& = \sum_{j=1}^n \bM^{\dagger}_{ij}\Big(d_j \bx^\gamma_j(\bw) - (\bN \bx^\gamma(\bw))_{\le j-1}\Big) - \frac{1}{n}\sum_{\ell=1}^n \Bigl(d_\ell \bx^\gamma_\ell(\bw) - (\bN \bx^\gamma(\bw))_{\le \ell-1}\Bigr)\sum_{j=1}^n \bM^{\dagger}_{ij}.
		\end{aligned}
	\end{equation*} Hence, in each iteration of the relocated fixed-point iteration, the evaluation of the $i$-th component of the third step in \eqref{e:var-FB} requires, in principle, the calculation of $\bx^{\gamma_k}_i(\bz_k)$ and $\bx^{\gamma_k}_i(\bw_k)$. In total, in each iteration one would need to evaluate $2n$ resolvents, while in a traditional fixed-point iteration, one would only need $n$ resolvent evaluations. 
\end{remark}

In order to overcome the issue described in Remark~\ref{r:cheap-reloc}, we use Lemma~\ref{l:known-J}\ref{l:known-J-1} and Remark~\ref{r:FPR} to construct a \Qname{} that does not require extra resolvent evaluations. This is formalised in the following proposition.

\begin{proposition}[Inexpensive \Qname{s}] \label{p:graph-FPR-cheap-gen} Under the assumptions of Proposition~\ref{p:FPR-Q}, 
    define for all $\bz \in X^m$ and $\delta,\gamma \in \R_{++}$,\begin{equation} \label{e:cheap-FPR}Q^\prime_{\delta\gets\gamma}\bz := \frac{\delta}{\gamma}\bz +  \left( 1 - \frac{\delta}{\gamma} \right) \bM^\dagger \left[\left(d_i -  \sum_{j=1}^{i-1}N_{ij}\right)_{i=1}^n\otimes \bx^\gamma_1(\bz)\right].
		\end{equation} Then $(Q^\prime_{\delta\gets \gamma})_{\gamma, \delta \in \R_{++}}$ 
	define  \Qname{s} for the distributed forward-backward operators $(T_{\theta,\gamma})_{\gamma\in\R_{++}}$ given by \eqref{e:FB-operator}-\eqref{e:FB-resolvent}, with  Lipschitz constants\begin{equation}\label{eq:cheap-Lip}
		    \mathcal{L}_{\delta\gets\gamma} = \max\left\{1,\frac{\delta}{\gamma} + \left|1 - \frac{\delta}{\gamma}\right| \frac{\|(M_{1j})_{j=1}^{m}\| \left\| M^\dagger\left(d_i -  \sum_{j=1}^{i-1}N_{ij}\right)_{i=1}^n\right\|}{d_1}\right\}.
		\end{equation}Furthermore, the following identity holds for $\bz \in X^m$ and $\delta,\gamma \in \R_{++}$: \begin{equation} \label{e:recycle}
		\bx_1^\gamma(\bz) = \bx_1^\delta(Q^\prime_{\delta\gets \gamma}\bz). 
	\end{equation}
\end{proposition}
\begin{proof}
Let $\delta,\gamma,\theta \in \R_{++}$.  From Proposition~\ref{l:known-e}\ref{l:known-e-ii}, given $\bz \in \Fix T_{\theta,\gamma}$, there exists $\Bar{x} \in X$ such that for all $i=1,\dots, n$, $\bx^\gamma_i(\bz) = \Bar{x}$. In particular, $\bx^\gamma_1(\bz) = \bx^\gamma_i(\bz)$. This last identity combined with the first line in \eqref{eq:e-interm} and Assumption~\ref{a:stand}\ref{a:stand_N} and~\ref{a:stand_triangular} yield for all $j=1, \dots, n$,\begin{equation*} \be^\gamma_j(\bz) = \Bigl(d_j  - \sum_{t=1}^{j-1} N_{jt} \Bigr) \bx^\gamma_1(\bz)  - \frac{1}{n}\sum_{\ell=1}^n \Bigl(d_\ell  - \sum_{t=1}^{j-1} N_{\ell t} \Bigr) \bx^\gamma_1(\bz) = \Bigl(d_j  - \sum_{t=1}^{j-1} N_{jt} \Bigr) \bx^\gamma_1(\bz). \end{equation*} In this manner, in view of \eqref{e:FPR-Q}  and \eqref{e:cheap-FPR}, for all $\bz \in \Fix T_{\theta,\gamma}$, and $i=1,\dots, n$, 
    \begin{equation*}
	    (Q_{\delta\gets \gamma}\bz)_i  = \tfrac{\delta}{\gamma}\bz_i +  \left(1 - \tfrac{\delta}{\gamma} \right)\sum_{j=1}^n M^{\dagger}_{ij}\be^\gamma_j(\bz) = \tfrac{\delta}{\gamma}\bz_i +  \left(1 - \tfrac{\delta}{\gamma} \right)\sum_{j=1}^n M^{\dagger}_{ij}\Bigl(d_j  - \sum_{t=1}^{j-1} N_{jt} \Bigr) \bx^\gamma_1(\bz) = (Q^\prime_{\delta\gets\gamma}\bz)_i.
	\end{equation*} Remark~\ref{r:FPR} thus yields the first desired result after establishing that $Q^\prime_{\delta\gets\gamma}$ is Lipschitz continuous. Indeed, from \eqref{e:FB-resolvent}, nonexpansiveness of the resolvent and the Cauchy-Schwarz inequality, it follows that \begin{equation*}
				\|\bx^\gamma_1(\bz) - \bx^\gamma_1(\bw)\| \leq \frac{\|(M_{1j})_{j=1}^{m}\|}{d_1}\|\bz-\bw\|,
			\end{equation*} and, since $\bM^\dagger \left[\left(d_i -  \sum_{j=1}^{i-1}N_{ij}\right)_{i=1}^n\otimes \bx^\gamma_1(\bz)\right] =  \left[M^\dagger\left(d_i -  \sum_{j=1}^{i-1}N_{ij}\right)_{i=1}^n\right]\otimes \bx^\gamma_1(\bz)$, it follows that \begin{equation*}
				\|Q^\prime_{\delta\gets\gamma}\bz-Q^\prime_{\delta\gets\gamma}\bw\| \leq \frac{\delta}{\gamma}\|\bz - \bw\| + \left|1 - \frac{\delta}{\gamma}\right| \frac{\|(M_{1j})_{j=1}^{m}\| \left\| M^\dagger\left(d_i -  \sum_{j=1}^{i-1}N_{ij}\right)_{i=1}^n\right\|}{d_1}\|\bz-\bw\|.
			\end{equation*} This proves that $Q^\prime_{\delta\gets\gamma}$ is Lipschitz continuous with  constant \eqref{eq:cheap-Lip}, concluding that $Q^\prime_{\delta\gets\gamma}$ is a \Qname{} for the operator $T_{\theta,\gamma}$. Next, to show \eqref{e:recycle}, observe  that we have \begin{equation*}
					\begin{aligned}
						\bx_1^\delta(Q^\prime_{\delta\gets \gamma}\bz) & = J_{\frac{\delta}{d_1}A_1}\left(\frac{1}{d_1}(\bM Q^\prime_{\delta\gets \gamma}\bz)_1\right)\\ & = J_{\frac{\delta}{d_1}A_1}\left(\frac{\delta}{\gamma}\left(\frac{1}{d_1}\bM\bz\right)_1 +  \frac{1}{d_1}\left( 1 - \frac{\delta}{\gamma} \right) \left( \bM \bM^\dagger\left[  \left(d_i -  \sum_{j=1}^{i-1}N_{ij}\right)_{i=1}^n\otimes \bx^\gamma_1(\bz) \right] \right)_1 \right)\\
                        & = J_{\frac{\delta}{d_1}A_1}\left(\frac{\delta}{\gamma}\left(\frac{1}{d_1}\bM\bz\right)_1 +  \frac{1}{d_1}\left( 1 - \frac{\delta}{\gamma} \right)  P_{\range(\bM)}\left[ \left(d_i -  \sum_{j=1}^{i-1}N_{ij}\right)_{i=1}^n\otimes \bx^\gamma_1(\bz)  \right]_1 \right)\\
						& = J_{\frac{\delta}{d_1}A_1}\left(\frac{\delta}{\gamma} \left(\frac{1}{d_1}\bM\bz\right)_1 +\left(1 - \frac{\delta}{\gamma}\right) \bx^\gamma_1(\bz)\right) = \bx^\gamma_1(\bz),\\
					\end{aligned}
				\end{equation*} where the first line follows from \eqref{e:FB-resolvent}, the second line is due to \eqref{e:cheap-FPR}, and the third line holds in view of the identity $P_{\range(\bM)} = \bM\bM^\dagger$ \cite[Proposition 3.30(ii)]{BC2017}. For the fourth equality, it follows from  \eqref{e:range-M} and Assumption~\ref{a:stand}\ref{a:stand_N} and~\ref{a:stand_triangular}, as these imply that \begin{equation*}
				    P_{\range(\bM)}\left[ \left(d_i -  \sum_{j=1}^{i-1}N_{ij}\right)_{i=1}^n\otimes \bx^\gamma_1(\bz)  \right]_1 = d_1 \bx^\gamma_1(\bz) - \frac{1}{n}\sum_{i=1}^n\left(d_i -  \sum_{j=1}^{i-1}N_{ij}\right)\bx^\gamma_1(\bz) = d_1 \bx^\gamma_1(\bz).
				\end{equation*} Finally, the last equality in the last line above is an application of Lemma~\ref{l:known-J}\ref{l:known-J-1}. This concludes the proof. 
\end{proof}

\begin{remark}[On the inexpensive nature of \eqref{e:cheap-FPR}] The advantage of the \Qname{} introduced in Proposition~\ref{p:graph-FPR-cheap-gen} is that no extra resolvent evaluations are needed in the resulting relocated method in comparison to the non-relocated counterpart, thus preventing an increase in the computational cost per (relocated) iteration. More precisely,  the iterates take the form in \eqref{e:var-FB} using $Q^\prime_{\gamma_{k+1}\gets\gamma_k}$ in \eqref{e:cheap-FPR} as \Qname{}. From \eqref{e:recycle}, it follows that \[\bx_{1,k+1} = \bx_1^{\gamma_{k+1}}(\bz_{k+1})  = \bx_1^{\gamma_{k+1}}(Q^\prime_{\gamma_{k+1}\gets\gamma_k}\bw_k) = \bx_1^{\gamma_k}(\bw_k).\] Hence, given $\bz_0 \in X^{m}$, the algorithm initialises by computing $\bx_{1,0} = \bx_1^{\gamma_0}(\bz_0)$, and then for each iteration $k = 0, 1, 2, ...$
    \begin{enumerate}
        \item Evaluates the resolvents $\bx_{i,k} = \bx_i^{\gamma_k}(\bz_k)$ for $i = 2, \dots, n$ and calculates $\bw_k$ as in the second line in \eqref{e:var-FB},
        \item Calculates $\bx_{1,k+1}= \bx_1^{\gamma_k}(\bw_k)$, and then computes $\bz_{k+1}$ from the third line in \eqref{e:var-FB}.
    \end{enumerate} In this manner, barring the initialisation, in each iteration only $n$ resolvents are computed (instead of $2n$ resolvents as described in Remark~\ref{r:cheap-reloc}).
\end{remark}

\section{Graph-based distributed forward-backward methods} \label{s:graph}

The general framework in \cite{dao2025general} includes as special cases several distributed splitting algorithms of forward-backward nature. In this section, we apply the results of Section~\ref{s:FB} to derive relocated extensions of the  graph-based forward-backward method introduced in \cite{aragon2025forward}, and  propose a variable stepsize extension of the three-operator Davis--Yin splitting method \cite{davis2017three}.

\subsection{Relocated forward-backward algorithms devised by graphs} \label{s:graph-FB}

Originally designed in \cite{bredies2024graph} for resolvent splitting methods, graph-based extensions of operator splitting methods define their iterates based on the topology of a given graph. In this section, we follow the presentation in  \cite{aragon2025forward} and its generalisation in \cite[Example 4.1]{dao2025general}.

Let $n \geq 2$ and $G = (\mN, \mE)$ be a  directed graph with set of nodes $\mN = \{1, \dots, n\}$, and set of arcs $\mE$ with $|\mE| \geq n-1$. We use the notation \begin{equation*}
	i \tto j \iff  (i,j) \in \mE.
\end{equation*} Throughout this section, suppose \begin{equation} \label{e:diG-a1}
	i \tto j \implies i <j.
\end{equation} Observe that \eqref{e:diG-a1} implies that there are no self-loops in the graph (for no node $i \in \mN$, $i\tto i$ holds true). For $i \in \mN$, $\kappa_i := |\{j \in \mN: i \tto j \text{~~or~~} j \tto i\}|$ denotes the \emph{degree} of node $i$. 
Likewise, $\kappa_i^+ := |\{j \in \mN:  j \tto i \}|$ denotes the \emph{in-degree} of $i$, and  $\kappa_i^- := |\{j \in \mN:  i \tto j \}|$ denotes the \emph{out-degree} of $i$. Naturally, for all $ i \in \mN$, $\kappa_i = \kappa_i^+ + \kappa_i^-$. We also denote $\text{Deg}(G) := \diag(\kappa_1, \dots, \kappa_n)$.

\begin{assumption} \label{assump:1}
	Suppose $G = (\mN, \mE)$ is a directed graph that \begin{enumerate}
		\item is connected:  for any two distinct nodes $i,j \in \mN$, there exist distinct nodes $i_1, \dots, i_r \in \mN$, such that for all $s= 0, \dots, r$, $i_s \tto i_{s+1}$ or $i_{s+1} \tto i_s$, where $i_0 = i$ and $i_{r+1} = j$, 
		\item is simple: there are no multiple arcs between any two distinct nodes. 
	\end{enumerate}
\end{assumption}

We also define the \emph{Laplacian} matrix of $G$, a (symmetric) positive semidefinite matrix, defined as: for $ i,j\in \mN$, $\text{Lap}(G)_{i,i} = \kappa_i$, $\text{Lap}(G)_{i,j} =  -1$ if $i \tto j$ or $j \tto i$, and $\text{Lap}(G)_{i,j} = 0$ otherwise.  In view of \cite[Proposition 2.1]{bredies2024graph}, there exists a matrix $S \in \R^{n \times (n-1)}$ such that $S^*$ is onto and $\text{Lap}(G) = SS^*$, the onto decomposition of $\text{Lap}(G)$. The \emph{incidence matrix} of $G$, denoted Inc$(G) \in \R^{n \times |\mE|}$, is defined as: for $i,j \in \mN$ and $e \in \mE$, $\Inc(G)_{i, e} = 1$ if $e = (i,j)$, $\Inc(G)_{i,e} = -1$ if $e = (j,i)$, $\Inc(G)_{i,e} = 0$ otherwise.   If $G$ is a tree, then  the onto decomposition of $\text{Lap}(G)$ is given by $S = \Inc(G)$. 

In the next result, we formulate the iteration operator of forward-backward algorithms devised by graphs from \cite{aragon2025forward}.

\begin{lemma} \label{l:construction}
	Let $G^\prime = (\mN, \mE^\prime)$ be a connected subgraph of $G$ with the same set of nodes $\mN$, and $S$ be an onto decomposition of $\text{Lap}(G^\prime)$. Let $\gamma,\theta \in \R_{++}$. The iteration operator $T_{\theta,\gamma}: X^{n-1} \to X^{n-1}$ of  forward-backward algorithms devised by graphs from \cite{aragon2025forward} is given by, for $\bz \in X^{n-1}$ and $i= 1, \dots, n-1$, \begin{equation} \label{e:graph-FB-operator}
		(T_{\theta,\gamma}\bz)_i = \bz_i -  \theta \sum_{j=1}^n \mathbf{S}_{ji}\bx^\gamma_j(\bz),
	\end{equation} where \begin{equation} \label{e:graph-resolvents}
		\left\{\begin{aligned}
			\bx_1^\gamma(\bz) &= J_{\frac{\gamma}{\kappa_1}A_1}\left(\frac{1}{\kappa_1}\sum_{j=1}^{n-1} \mathbf{S}_{1j}\bz_j\right)\\
			\bx_i^\gamma(\bz) &= J_{\frac{\gamma}{\kappa_i}A_i}\left(\frac{2}{\kappa_i}\sum_{j\tto i}\bx_j^\gamma(\bz) - \frac{\gamma}{\kappa_i}B_{i-1}\bx^\gamma_{h(i)}(\bz) + \frac{1}{\kappa_i}\sum_{j=1}^{n-1}\mathbf{S}_{ij}\bz_j\right)\text{~for~} i=2,\dots,n.
		\end{aligned}\right.
	\end{equation} 
\end{lemma}
\begin{proof} As shown in \cite[Example 4.1]{dao2025general}, the following selection of matrices satisfies Assumption~\ref{a:stand}: $M = S$ (satisfying Assumption~\ref{a:stand}\ref{a:stand_kerM} and~\ref{a:stand_neg_semidef} in view of \cite[eq. (5.3)]{atenas2025relocated} and \cite[eq. (34)]{dao2025general}, respectively), $N \in \R^{n\times n}$ defined by $N_{ij} = 1$ if $j\tto i$, and $N_{ij} = 0$ otherwise, and $D = \frac{1}{2} \diag(\kappa_i)_{i=1}^n$. For the methods introduced in \cite{aragon2025forward}, $|\mE| = p = n-1 $, so that we take $P^* = [0_{(n-1)\times 1} \;|\; I_{n-1} ]$, and defining $h: \mN\setminus\{1\} \to \mN\setminus\{n\}$ such that for all $i \in \mN$, $h(i)\tto i$, we take $R \in \R^{(n-1)\times n}$ as: for $i=1,\dots,n-1$, $R_{i,h(i+1)}= 1$, and $R_{ij}=0$ otherwise for any $j \in \mN$. Furthermore, Assumption~\ref{a:stand}\ref{a:stand_N} holds true in view of the handshaking lemma (see, e.g. \cite[Fact 5.1 (1)]{atenas2025relocated}: $\sum_{i=1}^n (\kappa_i - 2\kappa_i^+) = 0$):\begin{equation} \label{eq:consec-handshaking}
    \sum_{i=1}^n\left(\sum_{j=1}^n N_{ij} - d_i\right) = \sum_{i=1}^n\left(\sum_{\substack{j \tto i}} 1 - \frac{1}{2}\kappa_i\right) = \sum_{i=1}^n\left(\kappa^+_i - \frac{1}{2}\kappa_i\right) = 0.
\end{equation} In this manner, we recover the algorithmic framework introduced in \cite{aragon2025forward}, as described in \cite[Example 4.1]{dao2025general}. Indeed, since for all $i \in \mN$, $d_i = \frac{1}{2}\kappa_i$, then using the change of variables \begin{equation} \label{e:graph-change-of-variables}
		\gamma \gets 2\gamma, \; \theta \gets 2\theta, \; \bz \gets 2\bz,
	\end{equation} the iteration operator in \eqref{e:FB-operator}-\eqref{e:FB-resolvent} yields \eqref{e:graph-FB-operator}-\eqref{e:graph-resolvents}.\end{proof}

\begin{remark} \label{r:matrix-to-graph}
Due to Lemma~\ref{l:construction}, in the current setting of graph-based forward-backward methods, we can define already two (nontrivial) \Qname{s}: 
\begin{itemize}
    \item for all $\gamma,\delta \in \R_{++}$, $\bz \in X^{n-1}$, \begin{equation} \label{e:graph-FPR}
		\tilde{Q}_{\delta\gets \gamma}\bz := \tfrac{\delta}{\gamma}\bz +  \left(1 - \tfrac{\delta}{\gamma} \right)\mathbf{S}^{\dagger}\tilde{\be}^\gamma(\bz) \mbox{~where~} \tilde{\be}^\gamma(\bz) := P_{\range(\mathbf{S})}\left(( \kappa_i -  2\kappa_i^+) \bx^\gamma_i(\bz) \right)_{i=1}^n,
	\end{equation}
  given by Proposition~\ref{p:FPR-Q} and the change of variables in \eqref{e:graph-change-of-variables}, which recovers the same \Qname{} as \cite[Section 5]{atenas2025relocated}, originally devised for graph-based pure resolvent splitting methods \cite{bredies2024graph}, and 
\item the inexpensive version of \eqref{e:graph-FPR} introduced in Proposition~\ref{p:graph-FPR-cheap-gen}, defined for all $\gamma,\delta \in \R_{++}$, $\bz \in X^{n-1}$ as \begin{equation} \label{e:cheap-FPR-graph}Q^\prime_{\delta\gets\gamma}\bz := \frac{\delta}{\gamma}\bz +  \left( 1 - \frac{\delta}{\gamma} \right) \bS^\dagger\left[ \left(\kappa_i - 2\kappa^+_i\right)_{i=1}^n\otimes \bx^\gamma_1(\bz)\right],\end{equation} where we use the fact that $d_i -  \sum_{j=1}^{i-1}N_{ij} = \kappa_i - 2\kappa^+_i$ and the change of variables in \eqref{e:graph-change-of-variables}.
\end{itemize}
Moreover, convergence of relocated fixed-point iterations for operators in the form \eqref{e:graph-FB-operator}--\eqref{e:graph-resolvents} follows directly from Theorem~\ref{t:dist-FB-convergence} and Remark~\ref{r:Lip-const} for any of the two \Qname{s} in \eqref{e:graph-FPR} and \eqref{e:cheap-FPR-graph}. 
\end{remark}

As a consequence of Proposition~\ref{p:graph-FPR-cheap-gen}, in the following result we introduce different ``cheap'' \Qname{s} for three different graph structures.

\begin{corollary}[Graph-dependent fixed-point relocators] \label{p:graph-FPR-cheap}
	Given a graph $G = (\mN, \mE)$ satisfying \eqref{e:diG-a1} and Assumption~\ref{assump:1}, then the following hold.
	\begin{enumerate}	
		\item \label{p:inward-star} Inward star-shaped subgraph:  let $\widecheck{G} = (\mN, \widecheck{\mE})$,  $\widecheck{\mE} = \{(i,n): i = 1, \dots, n-1 \}$ be a subgraph of $G$. Then, the operator defined for all $\bz \in X^{n-1}$ and $\delta,\gamma \in \R_{++}$ by \begin{equation} \label{e:in-star-FPR}\widecheck{Q}_{\delta\gets\gamma}\bz := \frac{\delta}{\gamma}\bz +  \left( 1 - \frac{\delta}{\gamma} \right) {(\kappa_i - 2\kappa_i^+)_{i=1}^{n-1}} \otimes \bx^\gamma_1(\bz)
		\end{equation} is a \Qname{} for the corresponding operator $T_{\theta,\gamma}$, with Lipschitz constant $$\mathcal{L}_{\delta\gets\gamma} = \max\left\{1,\frac{\delta}{\gamma} + \left|1 - \frac{\delta}{\gamma}\right| \frac{1}{\kappa_1}\sqrt{ \sum_{i=1}^{n-1} (\kappa_i - 2\kappa_i^+)^2}\right\}.$$
		
		\item \label{p:outward-star} Outward star-shaped subgraph: let $\widehat{G} = (\mN, \widehat{\mE})$, $\widehat{\mE} = \{(1,i): i = 2, \dots, n \}$ be a subgraph of $G$. Then, the operator defined for all $\bz \in X^{n-1}$ and $\delta,\gamma \in \R_{++}$ by \begin{equation} \label{e:out-star-FPR}\widehat{Q}_{\delta\gets\gamma}\bz := \frac{\delta}{\gamma}\bz -  \left( 1 - \frac{\delta}{\gamma} \right) {(\kappa_i - 2\kappa_i^+)_{i=2}^n} \otimes \bx^\gamma_1(\bz)
		\end{equation} is a \Qname{} for the corresponding operator $T_{\theta,\gamma}$, with Lipschitz constant$$\mathcal{L}_{\delta\gets\gamma} = \max\left\{1,\frac{\delta}{\gamma} + \left|1 - \frac{\delta}{\gamma}\right| \frac{\sqrt{n-1}}{\kappa_1}\sqrt{\sum_{i=2}^n (\kappa_i - 2\kappa_i^+)^2}\right\}.$$
		
		\item \label{p:sequential} Sequential subgraph: let $\Bar{G} = (\mN, \Bar{\mE})$, $\Bar{\mE} = \{(i,i+1): i=1,\dots,n-1\}$ be a subgraph of $G$. Then,  the operator defined for all $\bz \in X^{n-1}$, $\delta,\gamma \in \R_{++}$ by \begin{equation} \label{e:seq-FPR}\Bar{Q}_{\delta\gets\gamma}\bz := \frac{\delta}{\gamma} \bz + \left(1-\frac{\delta}{\gamma}\right) \left(\sum_{j=1}^i\big(\kappa_j - 2\kappa_j^+\big)\right)_{i=1}^{n-1} \otimes \bx^\gamma_1(\bz)
		\end{equation} is a \Qname{} for the corresponding operator $T_{\theta,\gamma}$, with Lipschitz constant $$\mathcal{L}_{\delta\gets\gamma} = \max\left\{1,\frac{\delta}{\gamma} + \left|1 - \frac{\delta}{\gamma}\right|\frac{1}{\kappa_1}\sqrt{\sum_{i=1}^{n-1}\left(\sum_{j=1}^i \kappa_j - 2\kappa_j^+\right)^2}\right\}.$$
	\end{enumerate}
	Furthermore, for each pair $(Q^\prime,G^\prime) \in \{(\widecheck{Q},\widecheck{G}),(\widehat{Q},\widehat{G}),(\Bar{Q},\Bar{G})\}$, and for all $ \delta,\gamma, \theta\in \R_{++}$ and $\bz \in X^{n-1}$, the identity in \eqref{e:recycle} holds.
\end{corollary}

\begin{proof}
The result that the operators are \Qname{s} satisfying \eqref{e:recycle} follows from applying Proposition~\ref{p:graph-FPR-cheap-gen} and Remark~\ref{r:matrix-to-graph} with the linear operators $S$ described as follows.

	\begin{enumerate}

		\item[\ref{p:inward-star}] The inward star-shaped subgraph is a tree whose incidence matrix is given by, for all $i=1,\dots,n-1$, $\Inc(\widecheck{G})_{ii} = 1$,  $ \Inc(\widecheck{G})_{ni} = -1 $, and $0$ otherwise. Take $M = S = \Inc(\widecheck{G})$, so that  for all $i = 1, \dots, n-1$, 
    $S^\dagger_{ii} = 1-\frac1n$, and for all $j = 1, \dots, n$, $j \neq i$, $S^\dagger_{ij} = - \frac1n$. Hence, in \eqref{e:cheap-FPR-graph}, ${S}^\dagger (\kappa_i -  2\kappa_i^+)_{i=1}^n = (\kappa_i -  2\kappa_i^+)_{i=1}^{n-1}$ in view of the handshaking lemma, yielding the expression in \eqref{e:in-star-FPR} for the \Qname{}. To compute the Lipschitz constant, the expression in \eqref{eq:cheap-Lip} provides a direct explicit expression by noting that $\|(M_{1j})_{j=1}^{n-1}\| =\|(S_{1j})_{j=1}^{n-1}\| = 1$ and, similarly to the first two identities in \eqref{eq:consec-handshaking}, we can deduce that $\frac{1}{d_1}\left\| M^\dagger\left(d_i -  \sum_{j=1}^{i-1}N_{ij}\right)_{i=1}^n\right\| = \frac{1}{\kappa_1}\|(\kappa_i -  2\kappa_i^+)_{i=1}^{n-1}\|$, from where we conclude \ref{p:inward-star}.

    \item[\ref{p:outward-star}] The outward star-shaped graph is a tree whose incidence matrix is given by, for all $i=1,\dots, n-1$, $\Inc(\widehat{G})_{1i} = 1$ and $\Inc(\widehat{G})_{(i+1)i} = -1$, and $0$ otherwise. Take $M = S = \Inc(\widehat{G})$, so that, for all $i=1,\dots, n-1$, $S^\dagger_{i(i+1)} = \frac1n - 1$, and for all $j = 1, \dots, n$, $j \neq i+1$, $S^\dagger_{ij} = \frac1n$. Hence,  in \eqref{e:cheap-FPR-graph}, ${S}^\dagger (\kappa_i -  2\kappa_i^+)_{i=1}^n = -(\kappa_i -  2\kappa_i^+)_{i=2}^{n}$ in view of the handshaking lemma, yielding the expression in \eqref{e:out-star-FPR} for the \Qname{}. To compute the Lipschitz constant we can proceed similarly to item \ref{p:inward-star}: $\|(M_{1j})_{j=1}^{n-1}\| =\|(S_{1j})_{j=1}^{n-1}\| = \sqrt{n-1}$ and $\frac{1}{d_1}\left\| M^\dagger\left(d_i -  \sum_{j=1}^{i-1}N_{ij}\right)_{i=1}^n\right\| = \frac{1}{\kappa_1}\|(\kappa_i -  2\kappa_i^+)_{i=2}^{n}\|$.
    
    \item[\ref{p:sequential}] The sequential graph is a tree whose incidence matrix is given by, for all $i = 1, \dots, n-1$, $\Inc(\Bar{G})_{ii} = 1$, $\Inc(\Bar{G})_{(i+1)i} = -1$, and $0$ otherwise.  Take $M = S = \Inc(\Bar{G})$, so that, for all $i=1,\dots, n-1$, $S^\dagger_{ij} = 1 - \frac{i}{n}$ for $j=1,\dots,i$, and $S^\dagger_{ij} = -\frac{i}{n}$ for $j=i+1,\dots, n$. Hence, in \eqref{e:cheap-FPR-graph}, for $i=1,\dots,n-1$, $$\begin{aligned}
				    \big({S}^\dagger (\kappa_j -  2\kappa_j^+)_{j=1}^n\big)_i &= \left(1 - \frac{i}{n}\right)\displaystyle\sum_{j=1}^i (\kappa_j - 2\kappa_j^+)  - \frac{i}{n}\sum_{j=i+1}^n (\kappa_j - 2\kappa_j^+) \\
                    & = \sum_{j=1}^i (\kappa_j - 2\kappa_j^+)  - \frac{i}{n}\sum_{j=1}^n (\kappa_j - 2\kappa_j^+)  = \sum_{j=1}^i (\kappa_j - 2\kappa_j^+) ,
				\end{aligned} $$ where the last equality follows from the handshaking lemma, yielding the expression in \eqref{e:seq-FPR} for the \Qname{}. Next, the expression in \eqref{eq:cheap-Lip} provides the Lipschitz constant in \ref{p:sequential}, since $\|(M_{1j})_{j=1}^{n-1}\| =\|(S_{1j})_{j=1}^{n-1}\| = 1$ and $\frac{1}{d_1}\left\| M^\dagger\left(d_i -  \sum_{j=1}^{i-1}N_{ij}\right)_{i=1}^n\right\| = \frac{1}{\kappa_1}\left\|\left(\sum_{j=1}^i (\kappa_j - 2\kappa_j^+) \right)_{i=1}^{n-1}\right\|$.
           \end{enumerate}
 \end{proof}

\begin{remark}[Implementation of relocated graph-based forward-backward algorithms] \label{rema:cheap-Qname}
In light of Corollary~\ref{p:graph-FPR-cheap}, we can easily implement the iteration \eqref{e:var-FB}. For each subgraph structure, the recycling step can be performed as follows. For all $\kkk$,
	
	\begin{enumerate}
		\item\label{rema-i} for the inward star-shaped subgraph, $\bx_{1,k+1} = J_{\frac{\gamma_k}{\kappa_1}A_1}(\frac{1}{\kappa_1}\bw_{1,k})$,  and from \eqref{e:in-star-FPR}, \[\begin{aligned}
			\bz_{k+1} &= \frac{\gamma_{k+1}}{\gamma_k}\bw_k + \left(1 - \frac{\gamma_{k+1}}{\gamma_k}\right) {(\kappa_i - 2\kappa_i^+)_{i=1}^{n-1}} \otimes \bx^{\gamma_k}_1(\bw_k)\\& = \frac{\gamma_{k+1}}{\gamma_k}\begin{pmatrix}
				\bw_{1,k} \\ \bw_{2,k} \\ \vdots\\ \bw_{n-1,k}
			\end{pmatrix} + \left(1 - \frac{\gamma_{k+1}}{\gamma_k}\right) \begin{pmatrix}
				\kappa_1 \bx_{1,k+1} \\
				(\kappa_2 - 2\kappa_2^+)\bx_{1,k+1}\\
				\vdots \\ (\kappa_{n-1} - 2\kappa_{n-1}^+)\bx_{1,k+1}
			\end{pmatrix}.
		\end{aligned} \]
		\item for the outward star-shaped subgraph, $\bx_{1,k+1} = J_{\frac{\gamma_k}{\kappa_1}A_1}(\frac{1}{\kappa_1}\sum_{i=1}^{n-1}\bw_{i,k})$,  and from \eqref{e:out-star-FPR}, \[\begin{aligned}
			\bz_{k+1} &= \frac{\gamma_{k+1}}{\gamma_k}\bw_k - \left(1 - \frac{\gamma_{k+1}}{\gamma_k}\right) {(\kappa_i - 2\kappa_i^+)_{i=2}^{n}} \otimes \bx^{\gamma_k}_1(\bw_k)\\& = \frac{\gamma_{k+1}}{\gamma_k}\begin{pmatrix}
				\bw_{1,k} \\ \bw_{2,k} \\ \vdots\\ \bw_{n-1,k}
			\end{pmatrix} - \left(1 - \frac{\gamma_{k+1}}{\gamma_k}\right) \begin{pmatrix}
				(\kappa_2 - 2\kappa_2^+) \bx_{1,k+1} \\
				(\kappa_3 - 2\kappa_3^+)\bx_{1,k+1}\\
				\vdots \\ (\kappa_n - 2\kappa_n^+)\bx_{1,k+1}
			\end{pmatrix}.
		\end{aligned} \]
		\item\label{rema-iii} for the sequential subgraph, $\bx_{1,k+1} = J_{\frac{\gamma_k}{\kappa_1}A_1}(\frac{1}{\kappa_1}\bw_{1,k})$,  and from \eqref{e:seq-FPR}, we get \begin{equation*} %\label{e:seq-FPR}
			\Bar{Q}_{\delta\gets\gamma}^i\bz := \begin{cases}
				\frac{\delta}{\gamma}\bz_1 + \left(1 - \frac{\delta}{\gamma}\right) \kappa_1 \bx^\gamma_1(\bz),\text{~for~} i = 1, \\
				\Bar{Q}_{\delta\gets \gamma}^{i-1}\bz +\frac{\delta}{\gamma} (\bz_i - \bz_{i-1}) +  \left(1 - \frac{\delta}{\gamma}\right) (\kappa_i - 2\kappa_i^+) \bx^\gamma_1(\bz), \text{~for~} i = 2, \dots, n-1,
			\end{cases}
		\end{equation*} and thus \[\begin{aligned}
			\bz_{1,k+1} &= \frac{\gamma_{k+1}}{\gamma_k}\bw_{1,k} + \left(1 - \frac{\gamma_{k+1}}{\gamma_k}\right) \kappa_1 \bx_{1,k+1},
		\end{aligned} \] and for $i=2,\dots,n-1$, \[\bz_{i,k+1} = \bz_{i-1,k+1} + \frac{\gamma_{k+1}}{\gamma_k} (\bw_{i,k} - \bw_{i-1,k}) + \left(1 - \frac{\gamma_{k+1}}{\gamma_k}\right)(\kappa_i - 2\kappa_i^+) \bx_{1,k+1}. \] Observe that this equivalent formulation of the recycling steps induced by the sequential subgraph is a generalisation of the variable stepsize Malitsky--Tam method proposed in \cite[Section 5.2]{atenas2025relocated}, whose constant stepsize counterpart was introduced in \cite{MT23}.
	\end{enumerate}
    \end{remark}

\subsection{A variable stepsize forward-backward three-operator splitting method} \label{s:FPR-DY}

The Davis--Yin three-operator splitting method \cite{davis2017three} is a method of the forward-backward family aiming to solve the following problem: \begin{equation*}
	\text{find~} x \in X \text{~such that~} 0 \in A_1x + A_2x + Bx,
\end{equation*} where $A_1$ and $A_2$ are maximally monotone operators on $X$, and $B$ is a $\beta$-cocoercive operator. To recover the Davis--Yin iteration operator from either the inward star-shaped or sequential graphs in Corollary~\ref{p:graph-FPR-cheap}, set $n=2$, $\mN = \{1,2\}$ and $\mE = \mE^\prime = \{(1,2)\}$, so that $\kappa_1=\kappa_2=\kappa_2^+ = 1$ and $\kappa_1^+ = 0$. Then, the Davis--Yin iteration operator is defined by: for $\theta,\gamma \in \R_{++}$ and $z \in X$,\begin{equation*}
	T^{DY}_{\theta,\gamma}z := z - \theta  ( x_1^\gamma(z) - x_2^\gamma(z)), \text{~where~}
	\left\{\begin{aligned}
		x_1^\gamma(z) &:= J_{\gamma A_1}(z)\\
		x_2^\gamma(z) &:= J_{\gamma A_2}\big(  -z +  2 x_1^\gamma(z) - \gamma Bx_{1}^\gamma(z)\big).
	\end{aligned}\right.
\end{equation*} In view of Corollary~\ref{p:graph-FPR-cheap}, a \Qname{} for $T^{DY}_{\theta,\gamma}$ is given by:\begin{equation} \label{e:FPR-DY}
	\Bar{Q}_{\delta\gets\gamma} = \frac{\delta}{\gamma}\Id + \left(1 - \frac{\delta}{\gamma}\right)x_1^\gamma.
\end{equation} Observe that $\Bar{Q}_{\delta\gets\gamma}$ coincides with the \Qname{} of the Douglas--Rachford method from \cite[Lemma 4.7]{atenas2025relocated}. The resulting variable stepsize version of the three-operator Davis--Yin splitting method is presented in Algorithm~\ref{a:DY} (cf. \cite[Algorithm 1.1]{atenas2025relocated}).

\begin{algorithm}[!ht]
	\caption{Relocated Davis--Yin algorithm for finding a zero of $A_1+A_2+B$. \label{a:DY}}
	\SetKwInOut{Input}{Input}
	\Input{Choose $z_0 \in X$, and an initial stepsize $\gamma_0 \in \R_{++}$. }
	
	Set $x_{0} = J_{\gamma_0 A_1}z_0$. 
	
	\For{$k=0,1,2,\dots$}{
		Step 1. Intermediate step. Choose $\theta_k \in \R_{++}$, and compute
        \begin{equation} \label{DY:variable-stepsize-1}
			\left\{\begin{aligned}
				y_{k} & = J_{\gamma_k A_2}(2x_{k} - z_k - \gamma_k Bx_{k}) \\
				w_k & = z_k + \theta_k(y_{k} - x_{k}). 
			\end{aligned}\right.
		\end{equation}
		
		Step 2. Next iterate. Compute
		\begin{equation} \label{DY:variable-stepsize-2}
			\left\{\begin{aligned}
				&x_{k+1}  = J_{\gamma_k A_1}w_k \\
				&\text{the next stepsize $\gamma_{k+1} \in \R_{++}$}\\
				&z_{k+1}  = \dfrac{\gamma_{k+1}}{\gamma_k} w_k + \left( 1 - \dfrac{\gamma_{k+1}}{\gamma_k} \right) x_{k+1}. 
			\end{aligned}\right.
		\end{equation}
	}
\end{algorithm}

\begin{corollary}[Convergence of relocated Davis--Yin method]
	Let $A_1, A_2$ be two maximally monotone operators on $X$, and $B$ be a $\beta$-cocoercive operator on $X$, such that $\zer(A_1+A_2+B) \neq \varnothing$. Let $\Gamma \subseteq (0, \frac{2}{\beta})$, $(\gamma_k)_\kkk $ be a sequence in $\Gamma$, and $\Bar{\gamma} \in \Gamma$  satisfying \begin{equation} \label{a:gamma}
		\gamma_k \to \overline{\gamma} \text{~and~}
		\sum_\kkk (\gamma_{k+1}-\gamma_k)_+ < +\infty .
	\end{equation} Let $(\theta_k)_\kkk $ be a bounded sequence in $\R_{++}$ separated from zero, such that $\liminf_\kkk(2 - \beta\gamma_k - 2  \theta_k) >0$. For the sequences $(x_k)_\kkk$, $(y_k)_\kkk$, $(z_k)_\kkk$ and $(w_k)_\kkk$ generated by Algorithm~\ref{a:DY}, the following hold.
	\begin{enumerate}
		\item $(z_k)_\kkk$ and $(w_k)_\kkk$ converge to the same point $\overline{z} \in \Fix T^{DY}_{\overline{\theta},\overline{\gamma}}$ for all $\overline{\theta} \in \Theta$.
		\item $(x_k)_\kkk$ and $(y_k)_\kkk$ converge to the same point $\bar{x} = J_{\overline{\gamma}A_1}(\Bar{z}) \in \zer(A_1+A_2 +B)$.
	\end{enumerate}
\end{corollary}

\begin{proof}
	Let $(x_k)_\kkk$, $(y_k)_\kkk$, $(w_k)_\kkk$, $(z_k)_\kkk$ be the sequences generated by Algorithm~\ref{a:DY}. Consider the graph $G = (\mN, \mE)$ with $\mN = \{1,2\}$ and $\mE = \{(1,2)\}$, the sequential subgraph of Corollary~\ref{p:graph-FPR-cheap}\ref{p:sequential} with $\mE^\prime = \mE$, and the \Qname{} for this subgraph structure in Remark~\ref{rema:cheap-Qname}\ref{rema-iii}. From \eqref{DY:variable-stepsize-2}, for all $\kkk$,  $z_{k+1} = \Bar{Q}_{\gamma_{k+1}\gets\gamma_k} w_k$, $w_k = z_k - \theta_k \begin{bmatrix}
		1 & -1
	\end{bmatrix}[x_k \quad y_k]^\top$, where $x_k = \bx^{\gamma_k}_1(z_k)$ and $y_k = \bx^{\gamma_k}_2(z_k)$. Hence, the sequences conform to \eqref{e:var-FB} with $M =\Inc(G^\prime) = \begin{bmatrix}
		1 & -1
	\end{bmatrix}^\top$. Moreover, the operator in \eqref{e:FPR-DY} is Lipschitz continuous with constant $\Lip_{\delta\gets\gamma} = \frac{\delta}{\gamma} + \left| 1- \frac{\delta}{\gamma} \right|$. The result follows from Theorem~\ref{t:dist-FB-convergence}, by noting that, from Lemma~\ref{l:con-ave} with $P^* = [0 \quad 1]$, $R = [1 \quad 0]$ and $M^* = [1 \quad -1]$, there holds that  $\mu = \beta$ in this setting.
\end{proof}

\begin{remark}[On the selection of variable stepsizes] \label{r:stepsize} One of the conditions of the general convergence result in Theorem~\ref{t:dist-FB-convergence} is that \eqref{a:gamma} holds. In view of \cite[Lemma 4.1]{Lorenz-Tran-Dinh}, the following safeguard scheme constructs a sequence of stepsizes satisfying this condition. Given bounds $0<\gamma_{\min}\leq\gamma_{\max} < +\infty$, such that for some $\nu \in (0,1)$, $\gamma_{\max} = \nu \sup{\Gamma}$, and for a sequence $(\zeta_k)_{k \in \N} \subseteq (0,1]$, such that $\zeta_0 = 1$ and $\sum_{k \in \N} \zeta_k < +\infty$, before evaluating $z_{k+1}$ in \eqref{DY:variable-stepsize-2} at the end of iteration $k$ in Algorithm~\ref{a:DY}, for some $\tau_k \in [\gamma_{\min},\gamma_{\max}]$,  define \begin{equation} \label{eq:step-heuristic}
		\gamma_{k+1} = (1-\zeta_k)\gamma_k + \zeta_k \tau_k.
	\end{equation} Observe that this rule guarantees $(\gamma_k)_{k \in \N} \subseteq [\gamma_{\min},\gamma_{\max}]$. In the next section, we show some preliminary numerical experiments using the aforementioned rule.\end{remark}

\section{Numerical experiments} \label{s:numerical}

In this section, we apply the algorithms in Section~\ref{s:graph} using different heuristics to define the sequence of stepsizes $(\gamma_k)_\kkk$ to solve statistical variable selection problems. Following Remark~\ref{r:stepsize}, we iteratively construct the sequence of stepsizes $(\gamma_k)_\kkk$ using \eqref{eq:step-heuristic} by setting $\zeta_k = \frac{0.1}{(k+1)^{1.5}}$ for all $\kkk$, and pick $\tau_k = \proj_{[\gamma_{\min},\gamma_{\max}]}(t_k)$, where $t_k>0$ can take three different forms: \begin{equation} \label{eq:nonstat}
	t_k = \frac{\|x_{k+1}\|}{\|x_{k+1}-w_k\|} \text{~inspired by \cite[Remark~4.12]{atenas2025relocated}},
\end{equation} \begin{equation}\label{eq:accDY}
	t_k = \frac{-2\gamma_k^2\cdot \frac{2-1.99}{2L} +\sqrt{\gamma_k^4\cdot \frac{(2-1.99)^2}{L^2}+4\gamma_k^2}}{2} \text{~inspired by \cite[eq. (3.6)]{davis2017three}},
\end{equation} or \begin{equation} \label{eq:Halpern}
	t_k = \frac{1}{k+1} \text{~inspired by \cite{Wittmann1992Approximation}}.
\end{equation}

We first examine the three-operator case (Section~\ref{s:FPR-DY}), and then the five-operator case (Section~\ref{s:graph-FB}). The datasets and code used in this work to run the experiments are available at \href{https://github.com/fatenasm/relocated-forward-backward}{github.com/fatenasm/relocated-forward-backward}.

\subsection{Constrained LASSO}

Given a matrix $A \in \R^{q \times d}$ and a vector $b \in \R^q$, the LASSO \cite{tibshirani1996regression} problem seeks to select variables from a vector $x \in \R^d$ such that $Ax \approx b$. In order to choose the variables that explain the reconstruction of the vector $b$, the LASSO uses the $\ell_1$-norm to promote sparsity (variable selection) in the solution. In this section, our goal is to reconstruct a vector $x \in \R^d$ in a predetermined box, such that $Ax = b$ in the least squares sense. We model this problem as the constrained LASSO problem: \begin{equation} \label{eq:ct-LASSO}
	\min_{x\in\R^d} \phi(x) := \frac12\|Ax-b\|^2_2 + \lambda\|x\|_1 \text{~s.t.~} x \in [-u,u]^d,
\end{equation} where $\lambda >0$ is a regularisation parameter, and $u >0$ defines the bounds of the box constraints. Since this problem shows a three-operator structure, we can employ Algorithm~\ref{a:DY} to find a solution. Through this simple example, we will illustrate that larger constant stepsizes do not necessarily yield better numerical performance of Algorithm~\ref{a:DY}, justifying the need for splitting methods that can vary stepsizes throughout iterations. We set $A_1 = \lambda \partial \|\cdot\|_1$, $A_2 = N_{[-u,u]^d}$, and $B_1 = A^\top(A\cdot- b)$. Clearly, $B_1$ is Lipschitz continuous with constant $L = \lambda_{\max}(A^\top A)$, where $\lambda_{\max}(\cdot)$ is the maximum eigenvalue function. As $B_1$ is the gradient of a convex function, it is  $L$-cocoercive. We generate $A$ and $b$ randomly, and control the eigenvalues of $A$ to prevent the Lipschitz constant $L$ from exploding or becoming too small. We also choose the small regularisation parameter $\lambda = 10^{-3}$ and the box bound $ u = 50$. Recall that we use \eqref{eq:step-heuristic} and \eqref{eq:nonstat}--\eqref{eq:Halpern} to construct the stepsizes.

Figure~\ref{fig:LASSO} shows the results obtained by applying Algorithm~\ref{a:DY} to solve problem \eqref{eq:ct-LASSO}, where the initial point is drawn randomly, $\gamma_0 = \frac{1}{\beta}$. We test the three fixed-point relocator (FPR) stepsize regimes listed above, as well as three constant stepsize regimes: $\gamma = \frac{0.1}{L}$, $\gamma = \frac{1}{L}$ and $\gamma = \frac{1.99}{L}$. We report the relative error throughout the iterations of the sequences $(x_k)_\kkk$, $(y_k)_\kkk$, and their associated function values.

\begin{figure*}[h!]
	\centering
	\begin{subfigure}[b]{0.495\textwidth}
		\centering
		\includegraphics[width=\textwidth]{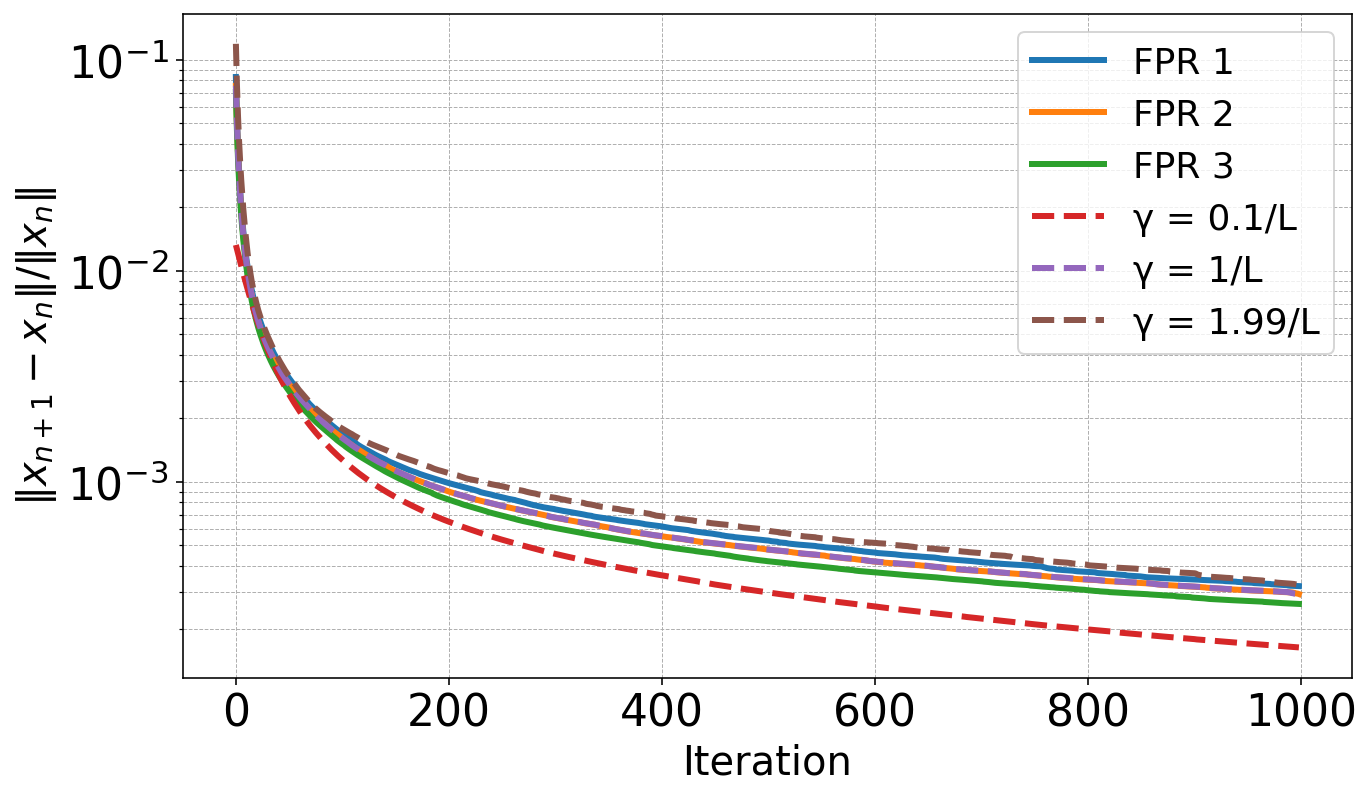}
		\caption[]%
		{Relative error of $(x_k)_\kkk$.}    
		\label{fig:LASSO-dx-1}
	\end{subfigure}
	\hfill
	\begin{subfigure}[b]{0.495\textwidth}  
		\centering 
		\includegraphics[width=\textwidth]{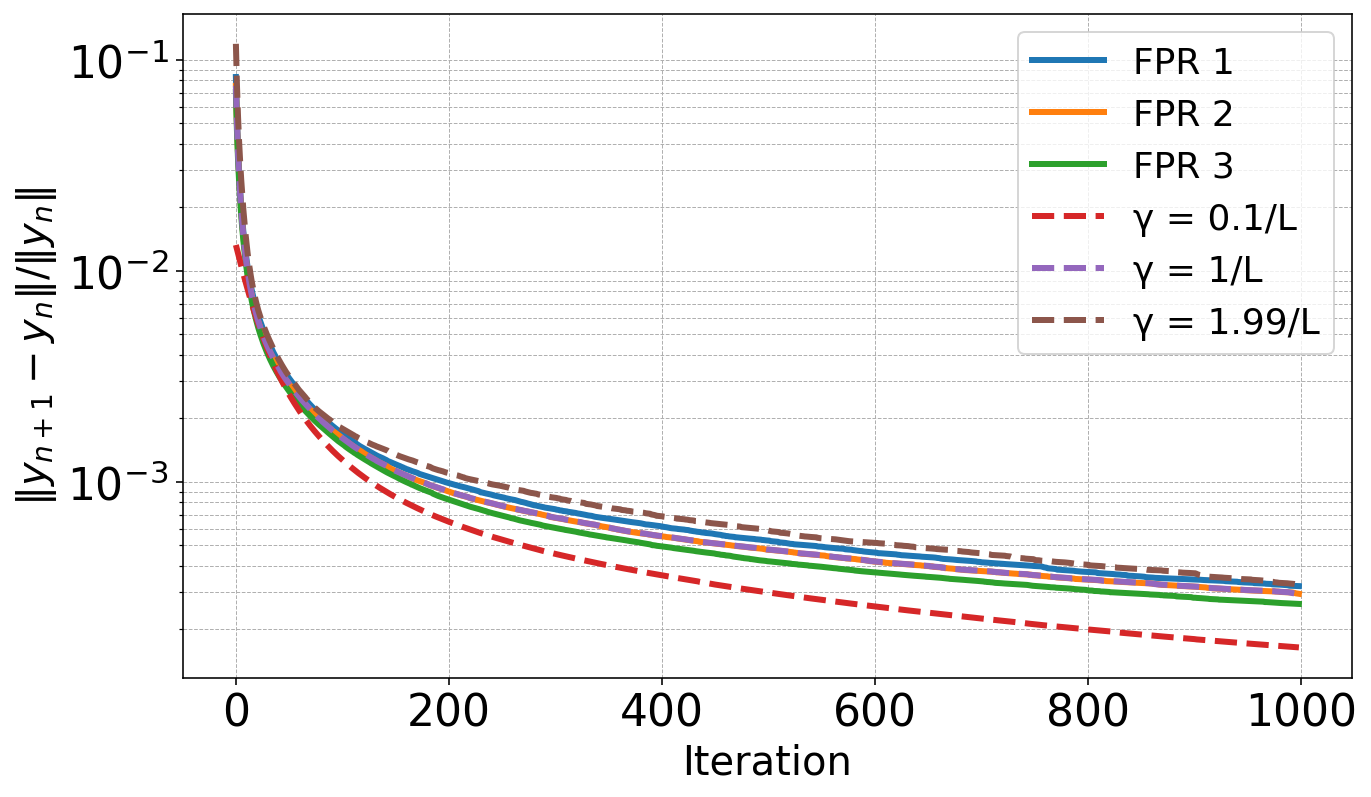}
		\caption[]%
		{Relative error of $(y_k)_\kkk$.}    
		\label{fig:LASSO-dx-2}
	\end{subfigure}
	\vskip\baselineskip
	\begin{subfigure}[b]{0.495\textwidth}   
		\centering 
		\includegraphics[width=\textwidth]{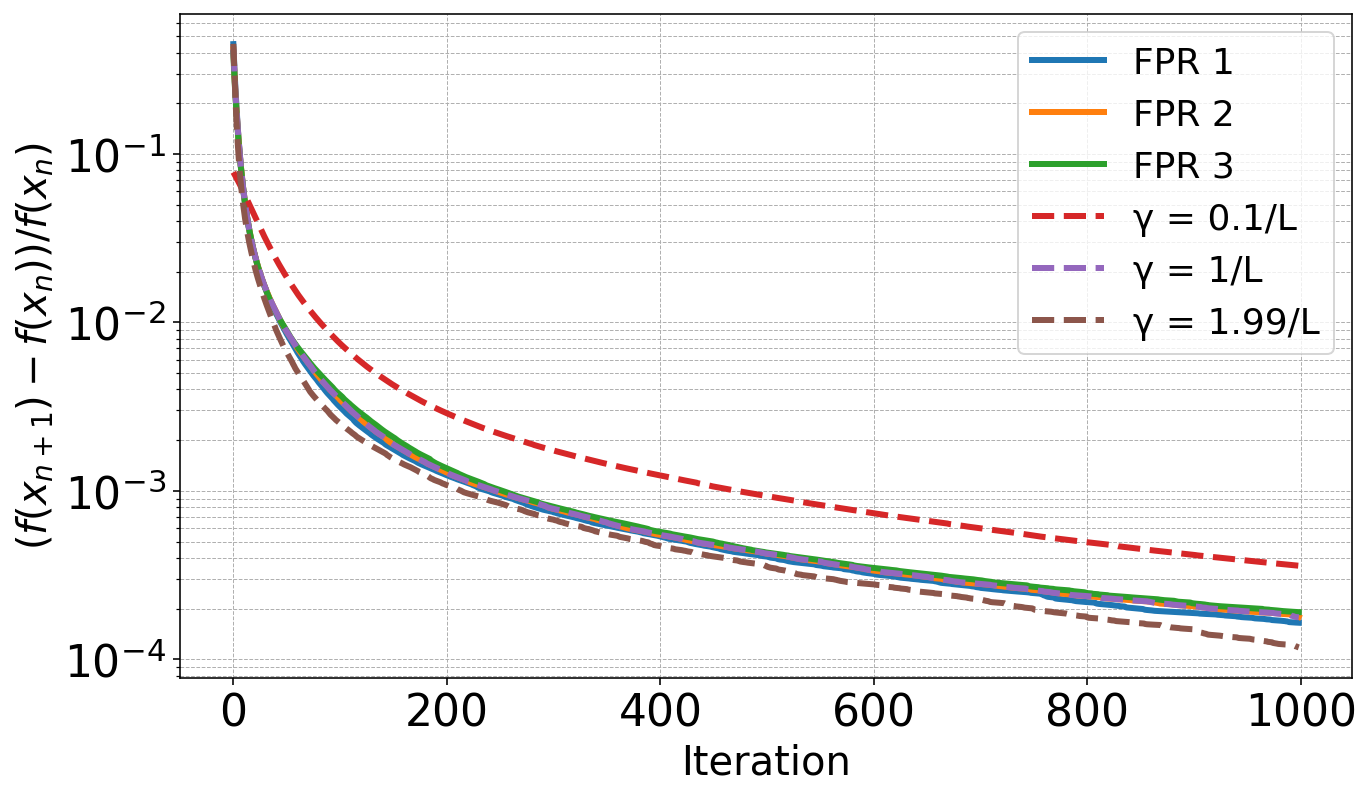}
		\caption[]%
		{Relative error of $(\phi(x_k))_\kkk$.}    
		\label{fig:LASSO-dfx-1}
	\end{subfigure}
	\hfill
	\begin{subfigure}[b]{0.495\textwidth}   
		\centering 
		\includegraphics[width=\textwidth]{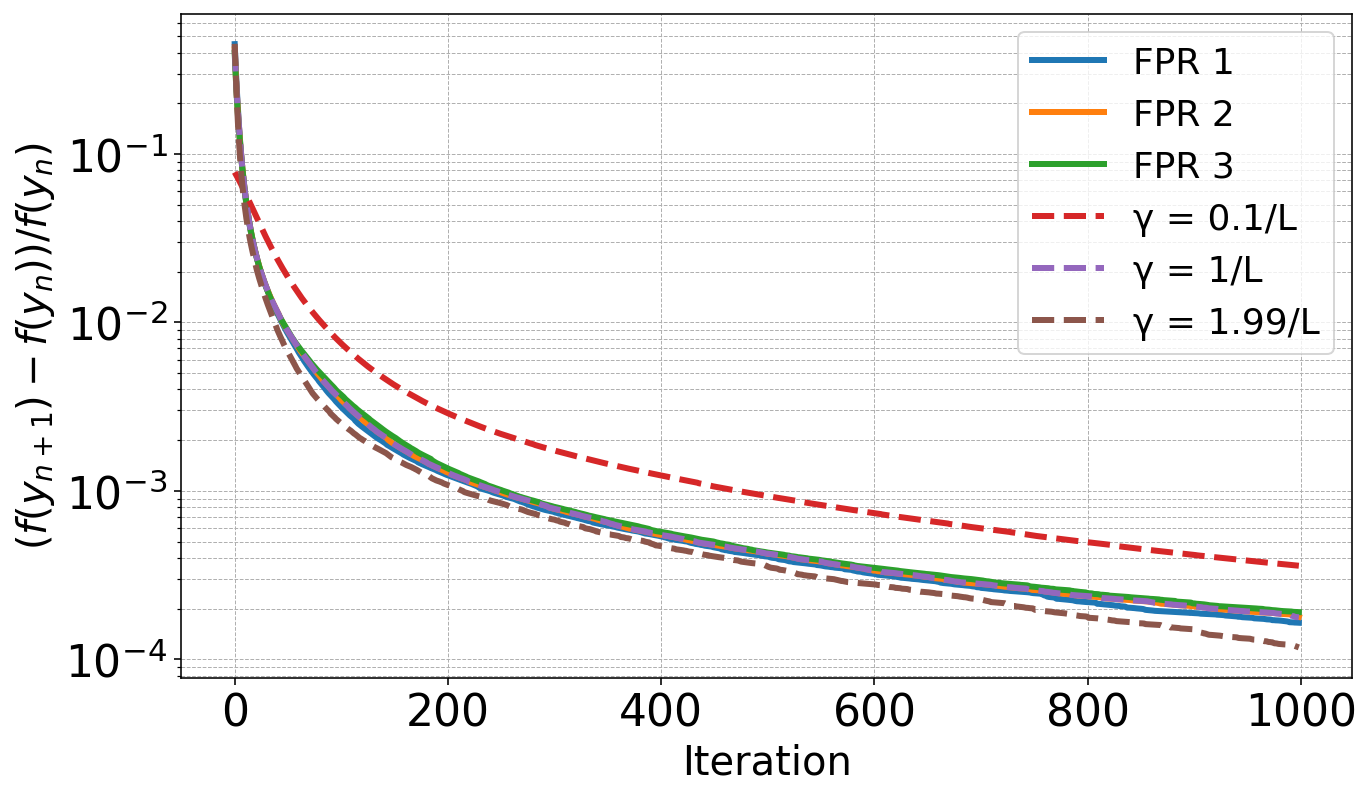}
		\caption[]%
		{Relative error of $(\phi(y_k))_\kkk$.}    
		\label{fig:LASSO-dfx-2}
	\end{subfigure}
	\caption[]
	{Relative error of iterates and function values generated by Algorithm~\ref{a:DY} for different stepsize regimes to solve problem \eqref{eq:ct-LASSO}, where FPR 1 is the stepsize rule \eqref{eq:nonstat}, FPR 2 is \eqref{eq:accDY} and FPR 3 is \eqref{eq:Halpern}. } 
	\label{fig:LASSO}
\end{figure*}

Figure~\ref{fig:LASSO} displays that better convergence results for the iterate relative error can be obtained for smaller values of constant $\gamma$, while better outcomes can be observed for larger values of $\gamma$ for the function values relative error. The FPR stepsize heuristics provide results in between the two extremes. These results align with the experiments performed in \cite{pedregosa2018adaptive} (e.g. Figures 1 and 5) for small regularisation parameters, where a larger stepsize does not consistently show a better performance for the objective function values. Here we only report the results for a small value of $\lambda$, as for larger values what is observed is that a better numerical performance can be obtained with larger stepsizes.

\subsection{Nonnegative elastic net}

The elastic net model \cite{zou2005regularization} is a generalisation of the LASSO for sparse linear regression in the presence of highly correlated features. In this section, we consider the elastic net problem with nonnegative variables. Given $A \in \R^{q\times d}$,  and $b \in \R^q$, the nonnegative elastic net problem is the following constrained and penalised least squares optimisation problem: \begin{equation} \label{e:elastic-net-obj} \min_{x \in \R^d} \phi(x) := \frac{1}{2}\|Ax-b\|_2^2 +  \lambda_1 \|x\|_1 + \frac{\lambda_2}{2}\|x\|_2^2 \text{~s.t.~} x \in [0, +\infty)^d,\end{equation} where $\lambda_1,\lambda_2 >0$ are regularisation parameters. The $\ell_2$-norm squared induces shrinkage and handles variable correlation. The combination of the two regularisation terms prevents overfitting to improve performance and reliability:  highly correlated predictors have similar coefficients. The nonnegative constraint in this case is necessary in applications where it is known a priori that the variables cannot be negative, e.g. in the air quality regression problem \cite{air_quality_360}, where the variables represent concentration of chemicals.

Adding a constraint in the elastic net problem in \eqref{e:elastic-net-obj} naturally yields a multioperator optimisation problem that can be handled by the forward-backward splitting methods in Section~\ref{s:graph-FB}. We choose $n=3$ and $p = 2$,  $A_1 = N_{[0,+\infty)^d}$, $A_2 = A_3 = \frac{\lambda_1}{2}  \partial \|\cdot\|_1$,  $B_1 = A^\top(A\cdot -b)$ and $B_2 = \lambda_2 \Id$. Since $B_1$ and $B_2$ are Lipschitz continuous and subdifferentials of convex functions, both are $\beta$-cocoercive with $\beta = \max\{\lambda_{\max}(A^\top A),\lambda_2\}$. We generate synthetic data to emulate the air quality dataset from \cite{air_quality_360}. In particular, we generate the matrix $A$ from an exponential distribution, and then impose linear dependency between some columns (to induce correlation between features). Furthermore, we also draw a vector $x^*$ from an exponential distribution and then set $b $ to be $Ax^*$ plus some random normal noise. In addition, we choose $\lambda_1 = \lambda_2 = 10^{-2}$ as regularisation parameters. Similarly to the previous section, we do not report the results for larger values of the regularisation parameters, as in those cases, better numerical performance is observed for larger stepsizes.

As in the previous section, we follow Remark~\ref{r:stepsize} to define the stepsizes for the relocated fixed-point iterations setting. In particular, we tested \eqref{eq:nonstat} and \eqref{eq:Halpern} in \eqref{eq:step-heuristic}, but decided to only display the results of the former, as there was no significant difference between the results using either of the options (observe that \eqref{eq:accDY} does not apply in a setting with more than three operators).

\begin{figure}[h!]
	\centering
	\subfloat[Relative residual $(\bx_k)_\kkk$,\newline inward star-shaped subgraph]{\includegraphics[width=0.33\textwidth]{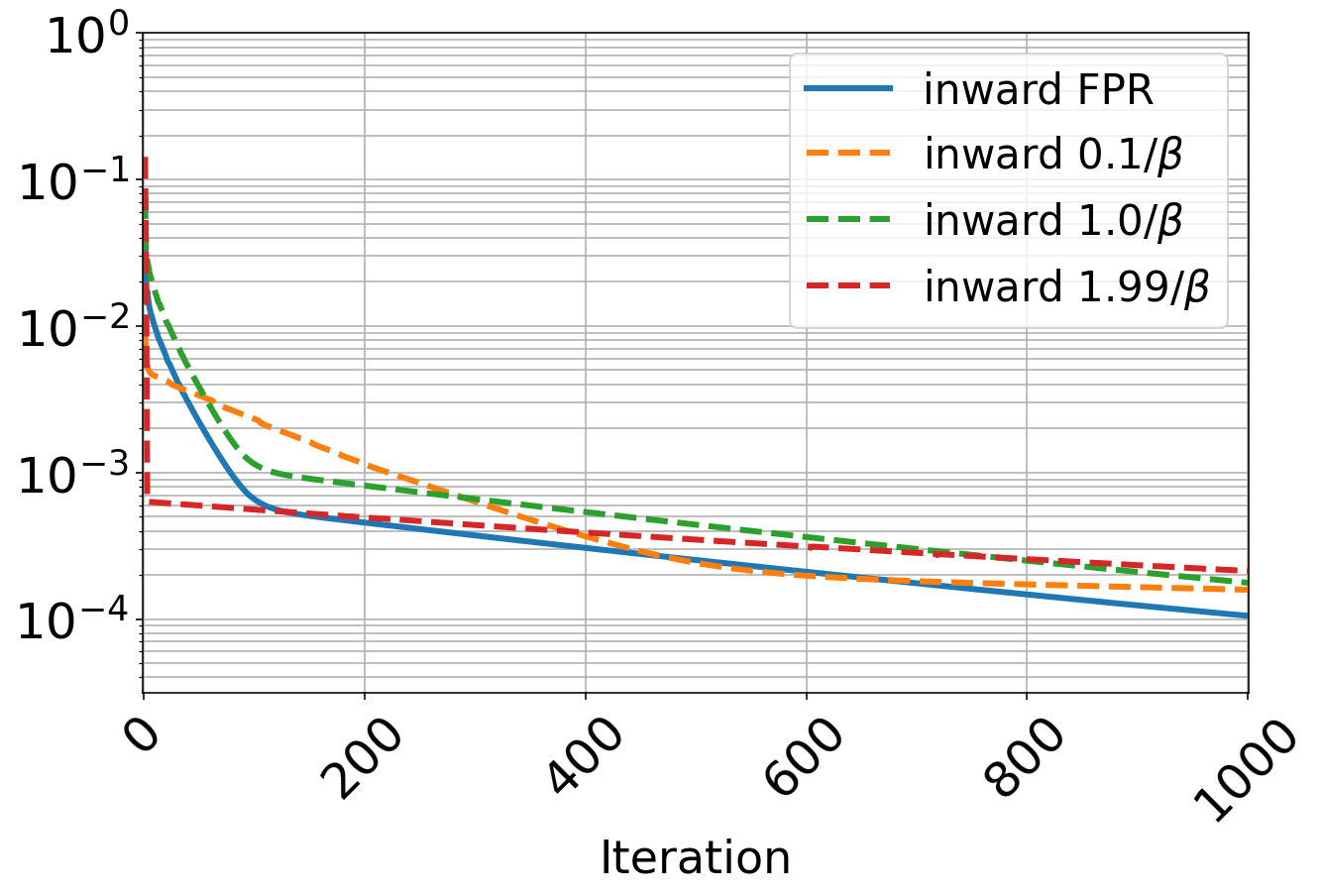}}\hfil
	\subfloat[Relative residual $(\bx_k)_\kkk$,\newline outward star-shaped subgraph]{\includegraphics[width=0.33\textwidth]{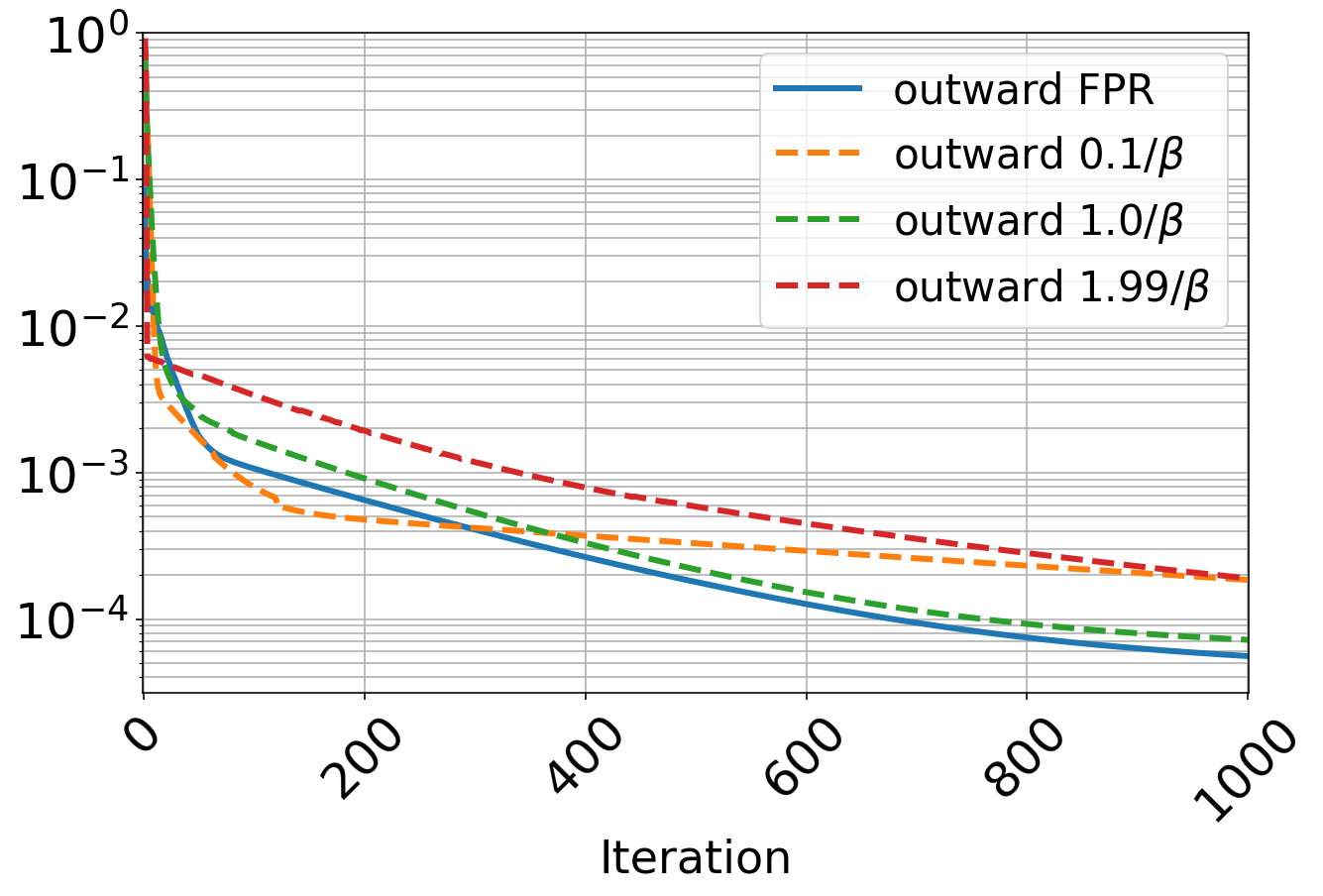}}\hfil 
	\subfloat[Relative residual $(\bx_k)_\kkk$,\newline sequential subgraph]{\includegraphics[width=0.33\textwidth]{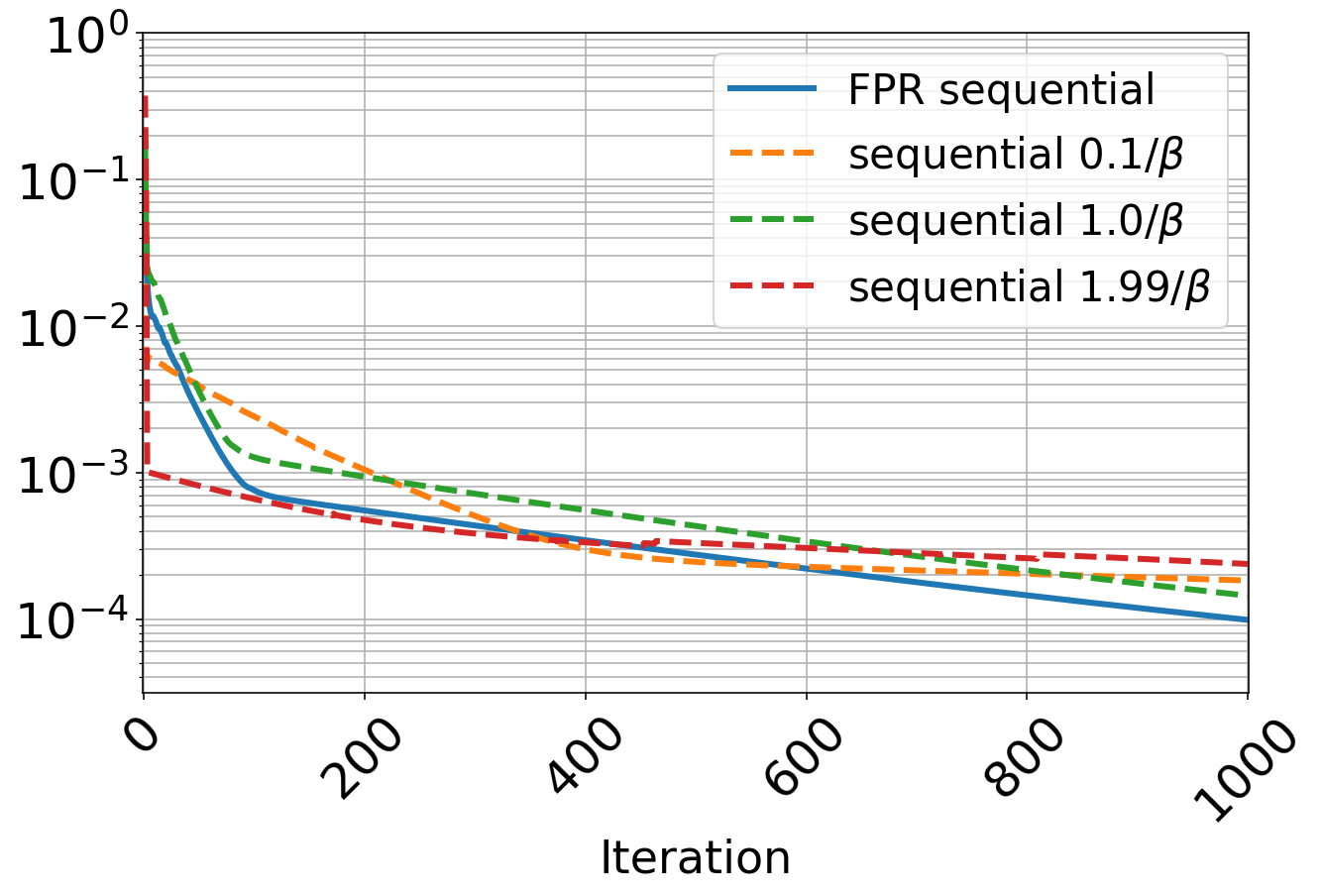}} 
	
	\subfloat[Relative residual $(\phi(\bx_k))_\kkk$,\newline inward star-shaped subgraph]{\includegraphics[width=0.33\textwidth]{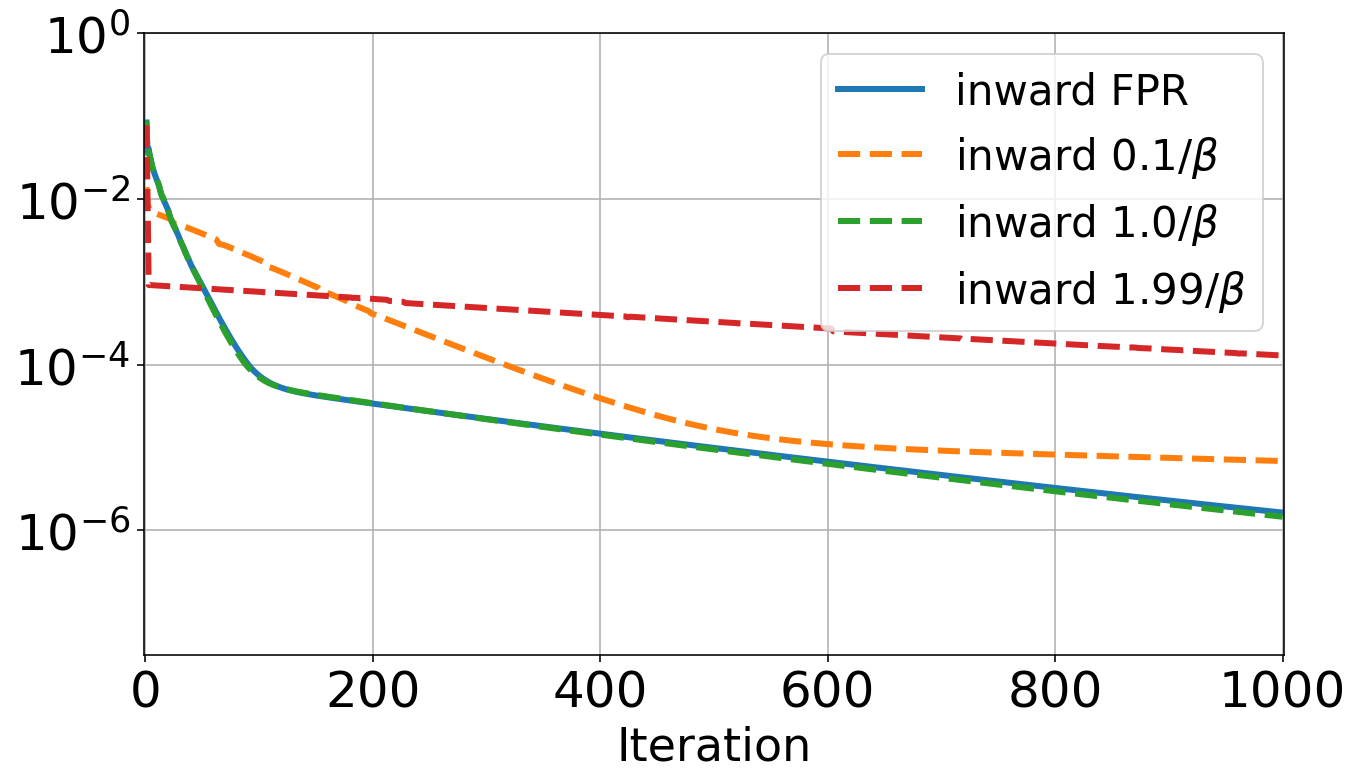}}\hfil   
	\subfloat[Relative residual $(\phi(\bx_k))_\kkk$,\newline outward star-shaped subgraph]{\includegraphics[width=0.33\textwidth]{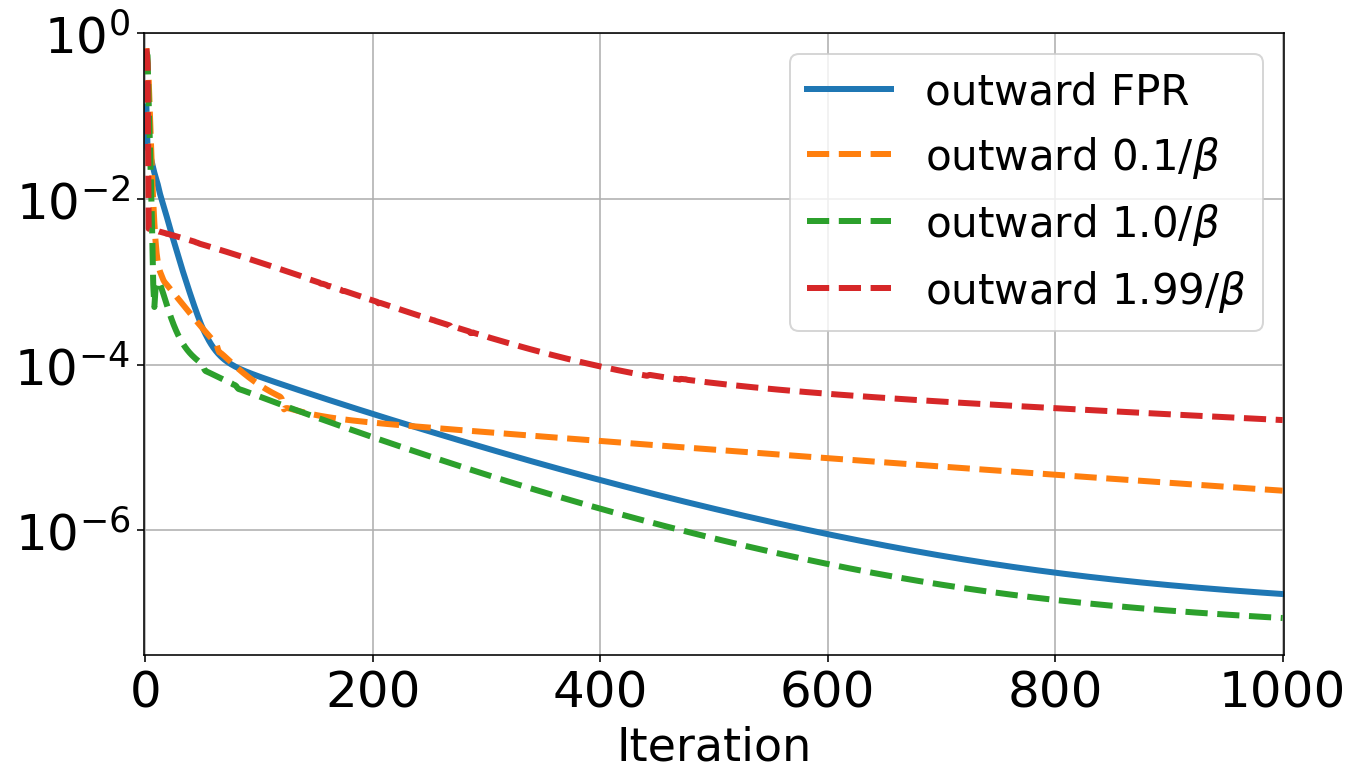}}\hfil
	\subfloat[Relative residual $(\phi(\bx_k))_\kkk$,\newline sequential subgraph]{\includegraphics[width=0.33\textwidth]{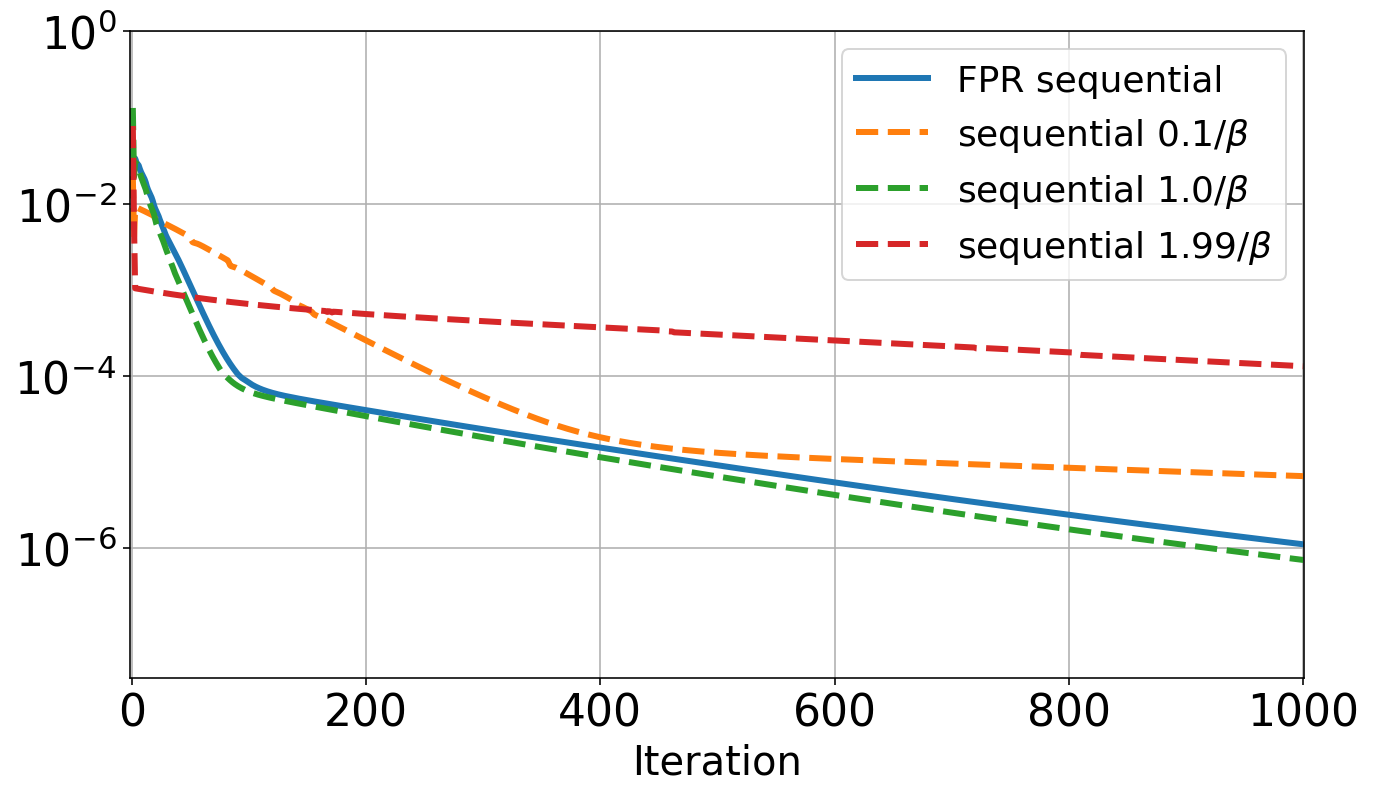}}
	\caption{Relative error of iterates and function values generated by \eqref{e:graph-FB-operator}--\eqref{e:graph-resolvents} for different stepsize regimes to solve problem \eqref{e:elastic-net-obj}, where FPR is the stepsize rule  \eqref{eq:nonstat}. }\label{fig:elastic}
\end{figure} Figure~\ref{fig:elastic} shows that, between the three regimes of constant stepsizes $\gamma = \frac{0.1}{\beta}$, $\gamma = \frac{1}{\beta}$, and $\gamma = \frac{1.99}{\beta}$, there are no consistent results for the best performance, although $\gamma = \frac{1}{\beta}$ is arguably the best choice. This behaviour is clearly an issue, since there is no a priori rule to choose a constant stepsize. Nevertheless, the heuristic to choose variable stepsizes is competitive in this setting. The authors expect that better consistent outcomes can be obtained with a rule defined to reduce the magnitude of the  maximum eigenvalue of the underlying iteration operator, akin to \cite{Lorenz-Tran-Dinh}.

\section{Conclusion} \label{s:conclusion}

In this work, we develop a framework of forward-backward methods with variable stepsize to solve multioperator monotone inclusions, based on the relocated fixed-point iterations theory. This type of method requires the definition of a so-called fixed-point relocator. We show that for forward-backward methods, these are closely related to the fixed-point relocators presented in \cite{atenas2025relocated} for pure resolvent splitting methods. Convergence guarantees are established, and graph-based implementations are also discussed with their corresponding fixed-point relocators.

The results obtained in this paper can be deemed the first step for designing efficient splitting methods for distributed problems. The second step would correspond to devising theoretically grounded stepsize rules that consistently display better numerical performance. As shown in the numerical experiments section, the relocated fixed-point iterations setting represents a promising alternative to tune the stepsize parameters automatically. Further research is required to fully discover the practical benefits of relocated fixed-point iterations.

		\bibliographystyle{abbrv}
\bibliography{biblio}
%\bibliography{samplebib-2}  % OJMO submission used this database
%\nocite{*}

\appendix

\section{Proof of Theorem~\ref{t:dist-FB-convergence}} \label{a:convergence}

We will divide the proof of Theorem~\ref{t:dist-FB-convergence} into steps. First, we show convergence of the relocated fixed-point iteration.

\begin{proposition} \label{p:aux-1}Under the assumptions of Theorem~\ref{t:dist-FB-convergence}, the following hold.
\begin{enumerate}
    \item \label{p:aux-1-i} $(\bz_k)_\kkk$ converges weakly to a point $\Bar{\bz} \in \Fix T_{\Bar{\theta},\Bar{\gamma}} $ for any $\Bar{\theta} \in \R_{++}$, and $\bz_k - T_{\theta_k,\gamma_k}\bz_k \to 0$.
    \item \label{p:aux-1-ii} For any $s \in (\R \bone)^\perp$, $\sum_{i=1}^n s_i \bx^{\gamma_k}_i(\bz_k) \to 0$.
\end{enumerate}\end{proposition}

\begin{proof}From \eqref{e:FB-operator}, for all $(\theta,\gamma) \in \Theta\times\Gamma$, $ T_{\theta,\gamma} = I -  \theta \bM^* \bx^\gamma$. Hence, the sequence $(\bz_k)_\kkk$ generated by \eqref{e:var-FB} corresponds to \eqref{e:z-iteration} with $\lambda_k = 1$ for all $\kkk$. From Lemma~\ref{l:con-ave}, $T_{\theta,\gamma}$ is conically $\eta(\theta,\gamma)$-averaged for $\eta(\theta,\gamma) = \frac{2\theta}{2-\gamma\mu}>0$. From the assumption on $(\gamma_k)_\kkk$, and $(\theta_k)_\kkk$, there exist $\theta_{\min}, \theta_{\max}  \in \R_{++}$ such that for all $\kkk$, $\theta_k \in [\theta_{\min}, \theta_{\max}]$, \begin{equation*}
    \liminf_\kkk \eta_k \geq \frac{2\theta_{\min}}{2 - \Bar{\gamma}\mu}>0,
\end{equation*} and \begin{equation*}
    \liminf_\kkk \frac{1-\eta_k}{\eta_k} = \liminf_\kkk \frac{1}{2\theta_k}(2 - \gamma_k\mu - 2\theta_k) \geq \frac{1}{2\theta_{\max}}\liminf_\kkk (2 - \gamma_k\mu - 2\theta_k)>0.
\end{equation*} Furthermore, for all $\gamma \in \Gamma$, $\theta_1,\theta_2 \in \Theta$, $\Fix T_{\theta_1,\gamma} = \Fix T_{\theta_2,\gamma} = \{\bz \in X^m: \bx^\gamma(\bz) \in \ker(\bM^*) \}$, and $\Fix T_{\theta_1,\gamma} \neq \emptyset$ from Proposition~\ref{l:known-e}\ref{l:known-e-0}. Since $(Q_{\delta\gets\gamma})_{\delta,\gamma \in \Gamma}$ are \Qname{s} of $(T_{\gamma})_{\gamma \in \Gamma}$ (Proposition~\ref{p:FPR-Q}), and from Lemma~\ref{l:known-J}\ref{l:known-J-2}, for all $\bz \in X^m$, $(\theta,\gamma) \mapsto T_{\theta,\gamma}\bz$ is continuous, then Theorem~\ref{t:FPR-convergence} yields the convergence of $(\bz_k)_\kkk$ and $(T_{\theta_k,\gamma_k}\bz_k)_\kkk$ to the same point $\Bar{\bz} \in \Fix T_{\Bar{\theta},\bar{\gamma}}$, and $\bz_k - T_{\theta_k,\gamma_k}\bz_k \to 0$.\end{proof}

The next result establishes key relations between the relocated fixed-point iteration and the resolvent steps.

\begin{lemma} \label{l:aux}
Suppose $A_1, \dots, A_n$ are  maximally monotone operators on $X$, and $B_1, \dots, B_{p}$ are  single-valued Lipschitz continuous with constant $\beta\in\R_{++}$, monotone operators. For every $i = 1, \dots, n$ and $x \in X$, set $ F_i x =  \sum_{j=1}^{\min\{i-1,p\}}P_{ij}B_jx$ from  \eqref{e:aux-eq}. Then, the following hold.
\begin{enumerate}
    \item \label{l:aux-max-mon} The operator $\mathcal{S}\colon X^n \rightrightarrows X^n$  given by
\begin{align*}
\mathcal{S} := \begin{bmatrix}
(A_1 + F_1)^{-1} \\
(A_2 + F_2)^{-1} \\
\vdots \\
(A_{n-1} + F_{n-1})^{-1} \\
A_n + F_n
\end{bmatrix} 
+\begin{bmatrix}
0 & 0 & \dots & 0 & -\Id \\
0 & 0 & \dots & 0 & -\Id \\
\vdots & \ddots & \vdots & \vdots & \vdots\\
0 & 0 & \dots & 0 & -\Id \\
\Id & \Id & \dots & \Id & 0
\end{bmatrix} 
\end{align*} is maximally monotone.

\item \label{l:aux-inclusion} For all $\bz \in X^m$ and $\gamma \in \R_{++}$, define \begin{equation*}
    \bp^{\gamma}(\bz) := \begin{bmatrix}
\bx_1^{\gamma}(\bz) -\bx_n^{\gamma}(\bz) \\
\bx_2^{\gamma}(\bz) -\bx_n^{\gamma}(\bz) \\
\vdots \\
\bx_{n-1}^{\gamma}(\bz) -\bx_n^{\gamma}(\bz) \\
\sum_{i=1}^n \frac{d_i}{\gamma}(\bu_i^{\gamma}(\bz) -\bx_i^{\gamma}(\bz)) +\sum_{i=1}^n \Delta_i^{\gamma}(\bz)
\end{bmatrix}, 
\end{equation*}  where\begin{equation*} 
    \left\{\begin{aligned}
        \bu_1^\gamma(\bz) &:= \frac{1}{d_1}(\bM\bz)_1\\
        \bu_i^\gamma(\bz) &:=  \frac{1}{d_i}(\bM\bz)_i + \frac{1}{d_i}(\bN\bx^\gamma(\bz))_{\le i-1} 
        \quad \text{~for~} i=2, \dots, n,
    \end{aligned}\right.
\end{equation*} and \begin{equation*}
    \Delta_i^{\gamma}(\bz) := F_i \bx_i^\gamma(\bz) - \big(\bP\bB\bR\bx^\gamma(\bz)\big)_{\leq\min\{i-1,p\}}.
\end{equation*}Also, define \begin{equation*}
    \bq^{\gamma}(\bz) := \begin{bmatrix}
\frac{d_1}{\gamma}(\bu_1^{\gamma}(\bz) -\bx_1^{\gamma}(\bz)) + \Delta_1^{\gamma}(\bz) \\
\vdots \\
\frac{d_{n-1}}{\gamma}(\bu_{n-1}^{\gamma}(\bz) -\bx_{n-1}^{\gamma}(\bz)) + \Delta_{n-1}^{\gamma}(\bz) \\
\bx_n^{\gamma}(\bz)
\end{bmatrix}.
\end{equation*} Then $\bp^{\gamma}(\bz)
\in
\mathcal{S}\left(\bq^{\gamma}(\bz)\right)$ and \begin{equation*}
    \frac{1}{\gamma}\sum_{i=1}^nd_i \bu_i^\gamma(\bz) - \frac{1}{\gamma}\sum_{i=1}^n\sum_{j=1}^{i-1}N_{ij}\bx_i^\gamma(\bz) = 0.
\end{equation*}
\end{enumerate}\end{lemma}

\begin{proof} For all $i = 1, \dots, n$, $A_i$ is maximally monotone, $F_i$ is maximally monotone with full domain (as it is cocoercive), and thus $A_n+F_n$ and $(A_i+F_i)^{-1}$ are maximally monotone \cite[Corollary 25.5(i), Proposition 20.22]{BC2017}. Since skew symmetric linear operators are maximally monotone with full domain \cite[Example 20.35]{BC2017}, then $\mathcal{S}$ is maximally monotone, concluding item \ref{l:aux-max-mon}. Next for item \ref{l:aux-inclusion}, from \eqref{e:FB-resolvent}, it follows that for all $i=1,\dots,n$, \[\frac{d_i}{\gamma}\big(\bu_i^\gamma(\bz)-\bx_i^\gamma(\bz)\big) - \big(\bP \bB \bR \bx^\gamma(\bz)\big)_{\le \min\{i-1,p\}} \in A_i\bx_i^\gamma(\bz).\] Adding $F_i\bx_i^\gamma(\bz)$ to both sides in the inclusion above and rearranging terms yields \[\bx_i^\gamma(\bz) \in (A_i+F_i)^{-1}\left(\frac{d_i}{\gamma}\big(\bu_i^\gamma(\bz)-\bx_i^\gamma(\bz)\big) + \Delta_i^\gamma(\bz)\right),\] and thus the inclusion $\bp^{\gamma}(\bz)
\in
\mathcal{S}\left(\bq^{\gamma}(\bz)\right)$ follows from the definition of the operator $\mathcal{S}$ and the vectors $\bp^{\gamma}(\bz)$ and $\bq^{\gamma}(\bz)$. For the second claim, the definition of $\bu^\gamma(\bz)=\big(\bu_1^\gamma(\bz),\dots,\bu_n^\gamma(\bz)\big)$ implies \begin{equation*}
    \sum_{i=1}^nd_i \bu_i^\gamma(\bz) - \big(N\bx^\gamma(\bz)\big)_{\leq i-1} = \sum_{i=1}^n (\bM \bz)_i = 0,
\end{equation*} where the second equality follows from \eqref{e:range-M} (cf.  the equation before \cite[eq. (22)]{dao2025general}).
\end{proof}

The next result states the asymptotic behaviour of the resolvent steps. 

\begin{proposition} \label{p:aux-2} Under the assumptions of Theorem~\ref{t:dist-FB-convergence}, let $(\bz_k)_\kkk$ and $(\bx_k)_\kkk$ be the sequences given in \eqref{e:var-FB}.  Define the sequences $(\bp^{\gamma_k}(\bz_k))_\kkk$ and  $(\bq^{\gamma_k}(\bz_k))_\kkk$ according to Lemma~\ref{l:aux}\ref{l:aux-inclusion}. Then the following hold.
\begin{enumerate}
    \item \label{p:aux-2-i} $(\bx^{\gamma_k}(\bz_k))_\kkk$ is bounded, and all $(\bx^{\gamma_k}_i(\bz_k))_\kkk$, for $i=1,\dots, n$, have the same weak cluster points.
    
    \item \label{p:aux-2-ii} $\bp^{\gamma_k}(\bz_k)\to 0$.
    
    \item \label{p:aux-2-iii} 
    Let $\bar{x}$ be a weak cluster point of   $(\bx^{\gamma_k}_1(\bz_k))_\kkk$, $\Bar{\bx} :=(\Bar{x},\dots,\Bar{x})\in X^n$, and for every $i = 1, \dots, n$, let $ F_i \bar{x}$ be as in \eqref{e:aux-eq}. Define \begin{equation*} 
    \left\{\begin{aligned}
        \Bar{u}_1 &:= \frac{1}{d_1}(\bM\Bar{\bz})_1\\
        \Bar{u}_i &:=  \frac{1}{d_i}(\bM\Bar{\bz})_i + \frac{1}{d_i}(\bN\Bar{\bx})_{\le i-1} \quad\text{~for~} i=2, \dots, n,
    \end{aligned}\right.
\end{equation*} where $\Bar{\bz}$ is the limit of $(\bz_k)_\kkk$ from Proposition~\ref{p:aux-1}\ref{p:aux-1-i}. Then, \begin{equation*}
    \Bar{q} := \begin{bmatrix}
\frac{d_1}{\Bar{\gamma}}(\Bar{u}_1 -\Bar{x}) \\ 
\vdots \\
\frac{d_{n-1}}{\Bar{\gamma}}(\Bar{u}_{n-1}-\bar{x}) \\
\Bar{x}
\end{bmatrix}
\end{equation*} is a weak cluster point of $(\bq^{\gamma_k}(\bz_k))_\kkk$.

\item \label{p:aux-2-iv} The sequence $(\bx_k)_\kkk$ converges weakly to a point $\Bar{\bx} = (\Bar{x}, \dots, \Bar{x}) \in X^n$, such that \eqref{e:x-zero} holds, and $\Bar{x}$ solves \eqref{e:distributed-monotone-inclusion}. 
    \end{enumerate}
\end{proposition}
\begin{proof}
Boundedness in item \ref{p:aux-2-i} follows from nonexpansiveness of the resolvent, Lipschitz continuity of $B_i$ for $i=1,\dots, p$, and using an inductive argument in \eqref{e:FB-resolvent}. Proposition~\ref{p:aux-1}\ref{p:aux-1-ii} with $s = e_i - e_j$, $i,j = 1, \dots, n$, $i \neq j$, where $e_i$ is the $i$-canonical vector in $\R^n$, yields \begin{equation} \label{e:dx-to-0}
    \bx_i^{\gamma_k}(\bz_k) - \bx_j^{\gamma_k}(\bz_k) \to 0.
\end{equation} Hence, claim \ref{p:aux-2-i}  follows. For item \ref{p:aux-2-ii}, combine \eqref{e:dx-to-0} and the definition of $\bp^{\gamma_k}(\bz_k)$ in \ref{l:aux-inclusion} to get $\bp_i^{\gamma_k}(\bz_k) \to 0$ for all $i =1, \dots, n-1$. For the last coordinate of $\bp^{\gamma_k}(\bz_k)$,  observe that from the definition of $\Delta_i^\gamma(\bz)$, for all $i = 1, \dots, n$, \begin{align*}
    \|\Delta_i^\gamma(\bz)\| &\leq \sum_{j=1}^{\min\{i-1,p\}} \left\| P_{ij}B_j\bx_i^{\gamma}(\bz) -  P_{ij}B_j\left(\sum_{t=1}^{j} R_{jt}\bx_t^{\gamma}(\bz)\right) \right\| \leq \|\bP\| \beta \sum_{j=1}^{\min\{i-1,p\}} \left\| \bx_i^{\gamma}(\bz) -  \sum_{t=1}^{j} R_{jt}\bx_t^{\gamma}(\bz) \right\| .
\end{align*} From Assumption~\ref{a:stand}\ref{a:stand_R} and~\ref{a:stand_triangular}, $\bx_i^\gamma(\bz) = \sum_{t=1}^{j-1} R_{jt}\bx_i^\gamma(\bz)$, then \begin{align*}
    \|\Delta_i^\gamma(\bz)\| &\leq \|\bP\| \beta \sum_{j=1}^{\min\{i-1,p\}} \left\| \sum_{t=1}^{j} R_{jt}(\bx_i^{\gamma}(\bz)-\bx_t^{\gamma}(\bz)) \right\|\leq \|\bP\| \|\bR\| \beta \sum_{j=1}^{\min\{i-1,p\}} \sum_{t=1}^{j}\|\bx_i^{\gamma}(\bz)-\bx_t^{\gamma}(\bz)\|,
\end{align*} and thus $\Delta_i^{\gamma_k}(\bz_k) \to 0$ follows from \eqref{e:dx-to-0}. Moreover, from Lemma~\ref{l:aux}\ref{l:aux-max-mon} and Assumption~\ref{a:stand}\ref{a:stand_N}, 
    \begin{align*}
        \bp_n^{\gamma_k}(\bz_k) &= \frac{1}{\gamma_k}\sum_{i=1}^n\left(\sum_{j=1}^{i-1}N_{ij}-d_i\right)\bx_i^{\gamma_k}(\bz_k) + \sum_{i=1}^n \Delta_i^{\gamma_k}(\bz_k) = \sum_{i=1}^n \Delta_i^{\gamma_k}(\bz_k) \to 0, \end{align*} from where \ref{p:aux-2-ii} follows.  Item \ref{p:aux-2-iii} is a direct consequence of the fact that $\Delta_i^{\gamma_k}(\bz_k) \to 0$ for all $i = 1, \dots , n$, Proposition~\ref{p:aux-1}\ref{p:aux-1-i}, and \ref{p:aux-2-i}, which guarantees the existence of a weak cluster point $\Bar{x}$ of $(\bx^{\gamma_k}_i(\bz_k))_\kkk$ for all $i = 1, \dots, n$. For \ref{p:aux-2-iv}, from Lemma~\ref{l:aux}\ref{l:aux-max-mon} and the fact that maximally monotone operators are demiclosed \cite[Corollary 20.38(ii)]{BC2017}, $\mathcal{S}$ is demiclosed. From Proposition~\ref{p:aux-2}\ref{p:aux-2-ii} and \ref{p:aux-2-iii}, taking the limit (up to a subsequence) in $\bp^{\gamma_k}(\bz_k) \in \mathcal{S}\left(\bq^{\gamma_k}(\bz_k)\right)$ (Lemma~\ref{l:aux}\ref{l:aux-inclusion}) gives $0 \in \mathcal{S}(\Bar{q})$. The latter inclusion and the definition of $\mathcal{S}$ imply \begin{equation*}
\frac{d_i}{\Bar{\gamma}}(\Bar{u}_i - \Bar{x}) \in A_i\Bar{x}+F_i\Bar{x} \quad \text{~for~} i = 1, \dots, n-1,
\end{equation*} and \begin{equation*}
   - \sum_{i=1}^{n-1}\frac{d_i}{\Bar{\gamma}}(\Bar{u}_i -\Bar{x}) \in A_n\Bar{x} +  F_n\Bar{x}.
\end{equation*} Hence, for  $i =1$, the definition of the resolvent yields \begin{equation*}
    \Bar{u}_1\in \left(\Id + \frac{\Bar{\gamma}}{d_1}A_1\right)\Bar{x} \iff \Bar{x} = J_{\frac{\Bar{\gamma}}{d_1}A_1}(\Bar{u}_1) = J_{\frac{\Bar{\gamma}}{d_1}A_1}\left(\tfrac{1}{d_1}(\bM \Bar{\bz})_1\right).
\end{equation*} Since $\Bar{\bz}$ is the unique limit of $(\bz_k)_\kkk$,  then all weak cluster points of $(\bx^{\gamma_k}_i(\bz_k))_\kkk$ coincide, and thus for all $i=1,\dots,n$,  $(\bx^{\gamma_k}_i(\bz_k))_\kkk$ converges weakly to $ J_{\frac{\Bar{\gamma}}{d_1}A_1}\left(\tfrac{1}{d_1}(\bM \Bar{\bz})_1\right)$, and $(\bx_k)_\kkk$ converges to $\Bar{\bx}=(\Bar{x}, \dots, \Bar{x}) \in X^n$. Furthermore, adding the $n$ inclusions above yields \begin{equation*}
   0 \in \sum_{i=1}^nA_i\Bar{x} +F_i \Bar{x},
\end{equation*} and since for all $i=1, \dots,n$, $ 1 = \sum_{j=1}^{n} P_{ji} = \sum_{j=1}^{\min\{i-1,p\}} P_{ji} = \sum_{j=1}^{p} P_{ji} $ due to Assumption~\ref{a:stand}\ref{a:stand_P} and \ref{a:stand_triangular}, then $\sum_{i=1}^n F_i \Bar{x} = \sum_{i=1}^n\sum_{j=1}^{\min\{i-1,p\}}P_{ij}B_j\Bar{x} = \sum_{j=1}^p B_j \Bar{x}$. Therefore, \begin{equation*}
    0 \in \sum_{i=1}^nA_i\Bar{x} + \sum_{j=1}^pB_j\Bar{x}.
\end{equation*}\end{proof}

\noindent\textit{Proof of Theorem~\ref{t:dist-FB-convergence}}.\ First, from Proposition~\ref{p:aux-1}\ref{p:aux-1-i}, $(\bz_k)_\kkk$ converges weakly to a point $\Bar{\bz} \in \Fix T_{\Bar{\theta},\Bar{\gamma}} $ for any $\Bar{\theta} \in \R_{++}$, and $\bz_k - T_{\theta_k,\gamma_k}\bz_k \to 0$. From \eqref{e:FB-operator} with $\lambda_k=1$ and \eqref{e:var-FB}, $\bw_k =  T_{\theta_k, \gamma_k}\bz_k$, and thus $\bz_k - \bw_k = \bz_k - T_{\theta_k, \gamma_k}\bz_k \to 0  $. Hence, $(\bw_k)_\kkk$ converges weakly to a point $\Bar{\bz}$ as well. The claims for $(\bx_k)_\kkk$ correspond to Proposition~\ref{p:aux-2}\ref{p:aux-2-iv}. \hfill$\square$

\end{document}